\documentclass[11pt,leqno]{amsart}
\usepackage{amsxtra,amscd,amsthm,amsmath,amssymb,mathtools}
\usepackage{longtable} 
\usepackage{bigstrut} 
\usepackage{enumerate}
\usepackage{verbatim}
\usepackage{graphicx}
\usepackage{xspace}
\usepackage{ifthen}
\usepackage{setspace}
\usepackage{tabularx}
\usepackage[all,cmtip,line]{xy}

\usepackage[usenames,dvipsnames]{color}

\usepackage{amsthm,amsfonts,amssymb,amsmath,amsxtra}
\usepackage[all]{xy}
\SelectTips{cm}{}
\usepackage{xr-hyper}
\usepackage[colorlinks=
   citecolor=Black,
   linkcolor=Red,
   urlcolor=Blue,
   backref=page]{hyperref}
\usepackage{verbatim}

 \usepackage{mathrsfs}

\RequirePackage{xspace}
\RequirePackage{etoolbox}
\RequirePackage{varwidth}
\RequirePackage{enumitem}
\RequirePackage{tensor}
\RequirePackage{mathtools}
\RequirePackage{longtable}
\RequirePackage{multirow}

\usepackage{scalerel}

\newcommand\reallywidehat[1]{\arraycolsep=0pt\relax%
\begin{array}{c}
\stretchto{
  \scaleto{
    \scalerel*[\widthof{\ensuremath{#1}}]{\kern-.5pt\bigwedge\kern-.5pt}
    {\rule[-\textheight/2]{1ex}{\textheight}} 
  }{\textheight} %
}{0.8ex}\\           
#1\\                 
\rule{-1ex}{0ex}
\end{array}
}

\newcommand{\sV}{\ensuremath{\mathscr{V}}\xspace}

\newcommand{\fkU}{\ensuremath{\mathfrak{U}}\xspace}

\newcommand{\diam}{{\Diamond}}

\newcommand{\BA}{\ensuremath{\mathbb {A}}\xspace}
\newcommand{\BB}{\ensuremath{\mathbb {B}}\xspace}
\newcommand{\BC}{\ensuremath{\mathbb {C}}\xspace}
\newcommand{\BD}{\ensuremath{\mathbb {D}}\xspace}

\newcommand{\BG}{\ensuremath{\mathbb {G}}\xspace}

\newcommand{\BL}{\ensuremath{\mathbb {L}}\xspace}
\newcommand{\BM}{\ensuremath{\mathbb {M}}\xspace}

\newcommand{\BP}{\ensuremath{\mathbb {P}}\xspace}
\newcommand{\BQ}{\ensuremath{\mathbb {Q}}\xspace}
\newcommand{\BR}{\ensuremath{\mathbb {R}}\xspace}

\newcommand{\BT}{\ensuremath{\mathbb {T}}\xspace}

\newcommand{\BZ}{\ensuremath{\mathbb {Z}}\xspace}

\newcommand{\calD}{\ensuremath{\mathcal {D}}\xspace}
\newcommand{\CE}{\ensuremath{\mathcal {E}}\xspace}

\newcommand{\CG}{\ensuremath{\mathcal {G}}\xspace}

\newcommand{\CK}{\ensuremath{\mathcal {K}}\xspace}

\newcommand{\CM}{\ensuremath{\mathcal {M}}\xspace}

\newcommand{\CO}{\ensuremath{\mathcal {O}}\xspace}
\newcommand{\CP}{\ensuremath{\mathcal {P}}\xspace}

\newcommand{\CU}{\ensuremath{\mathcal {U}}\xspace}
\newcommand{\CV}{\ensuremath{\mathcal {V}}\xspace}
\newcommand{\CW}{\ensuremath{\mathcal {W}}\xspace}

\newcommand{\CY}{\ensuremath{\mathcal {Y}}\xspace}

\newcommand{\an}{{\mathrm{an}}}

\DeclareMathOperator{\Gal}{Gal}
\newcommand{\GL}{\mathrm{GL}}

\DeclareMathOperator{\Spa}{Spa\,}
\DeclareMathOperator{\Spec}{Spec\,}
\DeclareMathOperator{\Spd}{Spd\,}
\DeclareMathOperator{\Spf}{Spf\,}

\newcommand{\Sp}{{\mathrm{Sp}}}

\newtheorem{theorem}{Theorem}
\newtheorem{proposition}[theorem]{Proposition}
\newtheorem{lemma}[theorem]{Lemma}

\newtheorem{corollary}[theorem]{Corollary}

\theoremstyle{definition}

\newtheorem{remark}[theorem]{Remark}

\newenvironment{altitemize}
   {\begin{list}
      {$\bullet$}
      {\setlength{\labelwidth}{0pt}
	   \setlength{\itemindent}{5pt}
       \setlength{\labelsep}{5pt}
       \setlength{\leftmargin}{0pt}
       \setlength{\itemsep}{\the\smallskipamount}
      }}
   {\end{list}}

\numberwithin{equation}{subsection}
\numberwithin{theorem}{subsection}

\setitemize[0]{leftmargin=0.2in,itemsep=\the\smallskipamount}
\setenumerate[0]{leftmargin=0.2in,itemsep=\the\smallskipamount}

\renewcommand{\to}{%
   \ifbool{@display}{\longrightarrow}{\rightarrow}%
   }
\let\shortmapsto\mapsto
\renewcommand{\mapsto}{%
   \ifbool{@display}{\longmapsto}{\shortmapsto}%
   }
\newlength{\olen}
\newlength{\ulen}
\newlength{\xlen}
\newcommand{\xra}[2][]{%
   \ifbool{@display}%
      {\settowidth{\olen}{$\overset{#2}{\longrightarrow}$}%
       \settowidth{\ulen}{$\underset{#1}{\longrightarrow}$}%
       \settowidth{\xlen}{$\xrightarrow[#1]{#2}$}%
       \ifdimgreater{\olen}{\xlen}%
          {\underset{#1}{\overset{#2}{\longrightarrow}}}%
          {\ifdimgreater{\ulen}{\xlen}%
             {\underset{#1}{\overset{#2}{\longrightarrow}}}
             {\xrightarrow[#1]{#2}}}}%
      {\xrightarrow[#1]{#2}}
   }
\makeatother
\newcommand{\xyra}[2][]{%
   \settowidth{\xlen}{$\xrightarrow[#1]{#2}$}%
   \ifbool{@display}%
      {\settowidth{\olen}{$\overset{#2}{\longrightarrow}$}%
       \settowidth{\ulen}{$\underset{#1}{\longrightarrow}$}%
       \ifdimgreater{\olen}{\xlen}%
          {\mathrel{\xymatrix@M=.12ex@C=3.2ex{\ar[r]^-{#2}_-{#1} &}}}%
          {\ifdimgreater{\ulen}{\xlen}%
             {\mathrel{\xymatrix@M=.12ex@C=3.2ex{\ar[r]^-{#2}_-{#1} &}}}
             {\mathrel{\xymatrix@M=.12ex@C=\the\xlen{\ar[r]^-{#2}_-{#1} &}}}}}%
      {\mathrel{\xymatrix@M=.12ex@C=\the\xlen{\ar[r]^-{#2}_-{#1} &}}}%
   }
\makeatletter
\newcommand{\xla}[2][]{%
   \ifbool{@display}%
      {\settowidth{\olen}{$\overset{#2}{\longleftarrow}$}%
       \settowidth{\ulen}{$\underset{#1}{\longleftarrow}$}%
       \settowidth{\xlen}{$\xleftarrow[#1]{#2}$}%
       \ifdimgreater{\olen}{\xlen}%
          {\underset{#1}{\overset{#2}{\longleftarrow}}}%
          {\ifdimgreater{\ulen}{\xlen}%
             {\underset{#1}{\overset{#2}{\longleftarrow}}}
             {\xleftarrow[#1]{#2}}}}%
      {\xleftarrow[#1]{#2}}
   }
\newcommand{\isoarrow}{%
   \ifbool{@display}{\overset{\sim}{\longrightarrow}}{\xrightarrow\sim}%
   }
   
\newcommand{\proet}{\mathrm{pro\acute{e}t}}

\newcommand{\quash}[1]{}

 \usepackage{relsize} 
\usepackage[bbgreekl]{mathbbol} 
\usepackage{amsfonts} 
\DeclareSymbolFontAlphabet{\mathbb}{AMSb} 
\DeclareSymbolFontAlphabet{\mathbbl}{bbold}

\newcommand{\crys}{{\rm crys}}
\newcommand{\br}{\breve}

\newcommand{\Sht}{\mathrm{Sht}}

\newcommand{\und}{\underline}

 \newcommand{\sk}{\vskip0.02in}
 
\begin{document}

\title[$p$-adic Borel extension for local Shimura varieties]{$p$-adic Borel extension for local Shimura varieties}

 \author{Abhishek Oswal}
\address{Mathematisches Institut, Albert-Ludwigs-Universität Freiburg, Freiburg, Germany}
\email{abhishek.oswal@math.uni-freiburg.de}

\author{Georgios Pappas}
\address{Department  of Mathematics, Michigan State University, E. Lansing, MI 48824, USA}
\email{pappasg@msu.edu}

\begin{abstract}
We show that the moduli spaces of Scholze's $p$-adic shtukas with framing satisfy a $p$-adic rigid analytic version of Borel's extension theorem.
In particular, this holds for  local Shimura varieties,  for all local Shimura data $(G,[b],\{\mu\})$, even for exceptional groups $G$, and  extends work of  Oswal-Shankar-Zhu-Patel who proved a $p$-adic Borel extension property for Rapoport-Zink spaces. As a corollary, we deduce 
 that all these spaces satisfy a $p$-adic rigid analytic version of Brody hyperbolicity.
\end{abstract}

 \date{\today}
\maketitle

\setcounter{tocdepth}{2}
\tableofcontents

\section{Introduction}

\subsection{The general set-up}
An important result in the foundations of the classical theory of Shimura varieties is Borel's extension theorem \cite[Thm A.]{Borel} for complex analytic maps of punctured polydisks into quotients of hermitian symmetric domains by (torsion-free) arithmetic groups.  This theorem implies that the algebraic structure given on these quotients by the Baily-Borel construction,  is uniquely determined by their complex analytic structure inherited from the symmetric domain, see \cite[3.10]{Borel}. Hence, the  structure of algebraic variety on these  quotients is ``canonical". 

Recently, Oswal-Shankar-Zhu-Patel \cite{OSZP} proved 
 a $p$-adic rigid analytic version of Borel's extension theorem for the analytification of Shimura varieties of abelian type. A crucial step in the course of their proof is a similar extension theorem for Rapoport-Zink spaces, which are, very roughly speaking, certain moduli spaces of $p$-divisible groups with additional structure. 
 From the modern point of view, Rapoport-Zink spaces are regarded as examples of \textit{local Shimura varieties}. In fact, Rapoport-Viehmann \cite{RV} conjectured the existence of general local Shimura varieties, as rigid analytic varieties which should be attached to group theoretic data and which should generalize Rapoport-Zink spaces. Soon afterwards, Scholze-Weinstein \cite{SW}
constructed these rigid analytic varieties as moduli spaces of \textit{$p$-adic shtukas}. These are objects which exist in the world of perfectoid spaces and whose definition is inspired by a construction of Drinfeld in the function field case. In perfectoid geometry, shtukas can be thought of as playing the role of generalized $p$-divisible groups; the reader is referred to \cite{SchRio} for an overview of the theory.

 The main purpose of this paper is to prove a similar extension result for all local Shimura varieties; this includes local Shimura varieties for the exceptional groups of type $E_6$ and $E_7$ for which there is no interpretation as moduli of $p$-divisible groups. Our arguments and results also extend to all the moduli spaces of $p$-adic shtukas ``with one leg" considered in \cite{SW}.

\subsection{Main results}
Suppose that $(G, [b],\{\mu\})$ are \textit{local Shimura data}, $K\subset G(\BQ_p)$, (see \S\ref{LSdata}),  and let $\CM_{G,b,\mu, K}$ be the corresponding local Shimura variety over ${\rm Sp}(\br E)$, where $\br E$ denotes the completion of the maximal unramified extension of the reflex field $E$ for the coweight $\mu$. Let $F$ be a finite field extension of $\br E$ and denote by $\BB=\BB_F$ the rigid analytic closed unit disk and by $\BB^\times=\BB-\{0\}$   the punctured closed unit disk over $F$. We show:  
 
 \begin{theorem}\label{mainThmIntro}
 Every rigid analytic morphism 
 $
 a^\times: \BB^{\times s}\times\BB^t \to \CM_{G,b,\mu, K}
$
 over $\Sp(\br E)$,  extends uniquely to   
 $
 a: \BB^{s+t} \to \CM_{G,b,\mu, K}
 $
  over $\Sp(\br E)$.
\end{theorem}
 
 This quickly implies that local Shimura varieties satisfy a certain  form of $p$-adic Brody hyperbolicity,  over discretely valued fields:

\begin{theorem}\label{Brody}
Every rigid analytic morphism  $Y^{\rm an}\to \CM_{G,b,\mu, K}$ over $\Sp(\br E)$, where $Y$ is a smooth irreducible rational curve over $F$, is constant.
\end{theorem}

In fact, we prove that,  for all \textit{local shtuka data}
$(G, [b],\{\mu\})$, $K\subset G(\BQ_p)$, and corresponding moduli spaces of shtukas $\Sht_{(G,b,\mu, K)}$, all morphisms  
$
\BB^{\times s}\times\BB^t \to  \Sht_{(G,b,\mu, K)}
$
extend to 
$
\BB^{s+t}\, \to  \Sht_{(G,b,\mu, K)}.
$
It follows that the moduli spaces $\Sht_{(G,b,\mu, K)}$ also satisfy the same form of $p$-adic Brody hyperbolicity, as above.
We also address the question of extending morphisms   $\BB^\times\times  X\to \Sht_{(G,b,\mu, K)}$, where $X$ is a general smooth rigid analytic variety over $F$, to $\BB\times X$; see Corollary \ref{mainThm0}  in the text. 

In general, when the coweight $\mu$ is not minuscule, the moduli spaces $\Sht_{(G,b,\mu, K)}$ are not rigid analytic varieties but only \textit{diamonds}, 
as defined in \cite{SchDiam}, \cite{SW}. 
In this case, the morphisms from $\BB^{\times s}\times\BB^t $ and $\BB^{s+t}$ as above make sense after applying the full embedding of the category of smooth rigid analytic  varieties over $\Sp(\br E)$ to the category of diamonds over $\Spd(\br E)$ given by the diamond functor $(\ )^\diam$. For example, $p$-adic Brody hyperbolicity is meant in the sense that
every  morphism  $Y^\diam=(Y^{\rm an})^\diam\to \Sht_{(G,b,\mu, K)}$ of diamonds over $\Spd(\br E)$, where $Y$ is a smooth irreducible rational curve over $F$, is constant, i.e. factors through a point $\Spd(F)\to \Sht_{(G,b,\mu, K)}$.
Following the same line of ideas, we also show:

\begin{theorem}\label{constantIntro}
Every morphism $S^\diam \to \Sht_{(G,b,\mu, K)}$ over $\br E$, where $S$ is a smooth irreducible quasi-projective variety over $F$ with {\sl finite} geometric \'etale fundamental group, is constant.
 \end{theorem}
 
The proof of this is simpler, and does not use the $p$-adic Borel extension theorem. The analogous statement for quotients of symmetric hermitian domains over the complex numbers is standard and, in many cases, can be derived as a consequence of the ``Theorem of the Fixed Part", see \cite[Thm. 3.1]{Vo}, \cite{DeligneHodgeII}.

 \subsection{Discussion of the proofs}   The proof of the main extension result Theorem \ref{mainThmIntro} follows a similar 
general strategy  
as the proof of the extension result of   \cite{OSZP} for Rapoport-Zink spaces but there are some crucial technical differences.
The argument here is streamlined and extended so it applies to general (non-minuscule) shtukas. There is no need to use,  or even mention, $p$-divisible groups and we can sidestep the reasoning regarding the period map
used in \cite[\S 5]{OSZP}. 
In fact, there is a significant difference when $s>1$ or $t>0$, which occurs even when $\mu$ is minuscule: The proof in \cite{OSZP}  uses an integral model of a global Shimura variety and the $p$-adic uniformization map in that case; such an integral model is, in general, unavailable. Here, this step is replaced by the use of a recent result of Guo-Yang \cite{GY}, as we explain below.
 
 The argument goes as follows: The diamond $\Sht_{(G,b,\mu, K)}$ is a moduli of shtukas ``with one leg bounded by $\mu$ and a framing by $b$" and we need to extend both the shtuka and its framing over the boundary. We first   reduce to the case that $G=\GL_n$ and $K=\GL_n(\BZ_p)$; this reduction is straightforward but crucially uses that we can allow $\mu$ to be non-minuscule. We then use that the shtuka over $\BB^{\times s}\times\BB^t $ corresponds to a de Rham \'etale local system. This local system is shown to be ``horizontally de Rham", in the sense that the vector bundle with connection associated to it by the $p$-adic Riemann-Hilbert correspondence (\cite{LiuZhu}, \cite{ScholzeHodge}) is constant. As in \cite{OSZP}, we now first use the logarithmic $p$-adic Riemann-Hilbert correspondence of Diao-Lan-Liu-Zhu  \cite{DLLZ} to extend the local system, and hence the shtuka, over the boundary to $\BB^{s+t}$. Next, we apply a result of Shimizu \cite{Sh} to deduce that the extended local system is crystalline at all classical points of $\BB^{s+t}$; moreover, at each classical point the corresponding Frobenius isocrystal is  given by $b$. The ``pointwise criterion  for crystallinity'' of \cite{GY} now implies that the local system is crystalline over $\BB^{s+t}$, in the sense of Faltings. Using this, we deduce that  the corresponding shtuka has constant Frobenius isocrystal given by $b$ and it admits a framing by $b$ over $\BB^{s+t}$. It remains to see that the framing of the shtuka over $\BB^{s+t}$ can be adjusted so that it extends the given framing over $\BB^{\times s}\times\BB^t$. This is shown by first proving that    positive Banach-Colmez spaces have only trivial sections over the rigid analytic variety $\BB^{\times s}\times\BB^t$. 
 
We give two proofs of $p$-adic Brody hyperbolicity, both using the extension theorem: For the first proof, we show that the big Picard theorem of $p$-adic analysis \cite{cherry} implies, roughly, that any rigid analytic variety (or, more generally, separated locally spatial diamond) which satisfies the extension property is hyperbolic in the above sense. For the second proof, 
we observe that the morphism with source the rational curve extends to $\BP^1$ and then apply Theorem \ref{constantIntro} to deduce that it has to be constant. 

 We conclude by remarking that these methods use crucially the assumption that the field $F$ is \emph{discretely valued}. For example, we do not know if all morphisms defined over $\BC_p$ from rational curves to local Shimura varieties have to be constant, see Remark \ref{variousHyperbolic}. 
 
 \subsection{Organization} In \S \ref{s:pre} we fix some notation and recall the definitions and basic properties of moduli of shtuka and the construction of various $p$-adic period rings for higher dimensional bases. We conclude \S \ref{s:pre} with Proposition \ref{propdeRham} which describes a relation between framed shtuka and \'etale local systems over rigid analytic varieties. The main results are stated and proven in \S \ref{s:thms}. We conclude  in \S \ref{s:questions} with some questions.

\medskip

\noindent \textit{Acknowledgements:} We are grateful to the referee for their careful reading of the manuscript and for several useful suggestions. G.P. is partially supported by NSF grant \#DMS-2100743. 

\section{Definitions and preliminaries}\label{s:pre}

\subsection{General assumptions and notation}

We use $F$ to denote a complete discrete valued field with valuation ring $F^\circ=\CO_F$. We will assume that $F$ is a finite extension of $\BQ_p$ or of $\br\BQ_p=W(k)[1/p]$, where $k$ is an algebraic closure of ${\mathbb F}_p$. We will choose an algebraic closure $\overline F$ of $F$ and 
write $\br F$ for the completion of the maximal unramified extension of $F$ in $\overline F$. We denote the completion of $\overline F$
by $C=\widehat{\overline F}$ or $\BC_p$.

Throughout this paper, we use the language of adic spaces, perfectoids, $v$-sheaves, diamonds, etc., as in \cite{SW}, \cite{SchDiam}, \cite{FS}. The reader is referred to these sources for the notations and details  of the constructions. 

We regard rigid analytic varieties and spaces as adic spaces.  We will freely use the equivalence of categories between quasi-separated rigid analytic spaces over $F$ in the sense of Tate and quasi-separated adic spaces locally of finite type over $\Spa(F, \CO_F)$ (\cite[Prop. 4.5]{Huber}). Given a rigid analytic space $\Sp(R)$ (in the sense of Tate) associated to an affinoid $F$-algebra $R$, the corresponding adic space  is  $\Spa(R,R^+)$ 
 with $R^+=R^\circ$ the subring  of power-bounded elements of $R$. 
 
 For an algebraic variety $S$ over $F$, we set $S^{\rm an}$ for its {\sl analytification}, i.e. the corresponding rigid analytic variety over $\Sp(F)$. We set $S^\diam$ for the diamond over $\Spd(F)$ obtained by applying the diamond functor \cite[Lect. 10]{SW} to the analytification $S^{\rm an}$ of $S$, i.e. $S^\diam:=(S^{\rm an})^\diam$.  

We denote by $F\langle t_1,\ldots , t_m\rangle$ the Tate algebra over $F$ in the variables $t_1,\ldots, t_m$. We let $\BB_F=\Spa(F\langle t \rangle, \CO_F \langle t \rangle)$ denote the closed unit disk over $F,$ and $\BB^\times_F$ denote the punctured closed unit disk\footnote{In \cite{OSZP}, these are denoted by $\BD_F$ and $\BD^\times_F$, respectively.}  over $F$:  If $\BA_n= \Spa(F\langle t, p^n/t \rangle, \CO_F\langle t, p^n/t \rangle)$ is the affinoid annulus with inner radius $|p|^n$ and outer radius $1$, then $\BB^\times_F=\bigcup_{n\geq 1}  \BA_n $.
We also have $\BB^m_F=\Spa(F\langle t_1,\ldots , t_m\rangle, \CO_F\langle t_1,\ldots , t_m\rangle)$. 

If $H$ is a locally profinite group,  we denote by $\underline H$  the pro-\'etale sheaf of groups on perfectoid spaces given by $\underline{H}(S)=C^0(|S|, H)$, where $C^0$ stands for continuous functions. 

We denote the Frobenius by $\phi$ or ${\rm Frob}$. In fact, we warn the reader that we will often abuse notation and use the same letter $\phi$ to denote different maps induced by Frobenius; hopefully, this will not cause too much confusion.

\subsection{Moduli of shtukas and local Shimura varieties}

Here we introduce the moduli spaces of $p$-adic shtukas, and explain some of their properties, following \cite{SW}.

\subsubsection{}\label{LSdata}
By definition, \emph{``local shtuka data''} are triples $(G, [b], \{\mu\})$ (\cite[\S 24.1]{SW}) where,  
\begin{itemize}
\item{} $G$ is a connected reductive group over $\BQ_p$,
\item{} $[b]$ is a $\phi$-conjugacy class of an element $b\in G(\br\BQ_p)$,
\item{} $\{\mu\}$ is a $G(\bar\BQ_p)$-conjugacy class of a cocharacter $\mu: {\BG_m}_{\bar\BQ_p}\to G_{\bar\BQ_p}$; \end{itemize}

We also assume the condition $b\in B(G, \mu^{-1})$ (see \cite[Def. 24.1.1]{SW}). 
(If this condition is not satisfied the corresponding local shtuka moduli spaces are empty.)
For simplicity, we will often write $(G, b,\mu)$ instead of   $(G, [b], \{\mu\})$.
As usual, we let $E\subset \bar\BQ_p$ be the ``reflex field'' of $(G,\mu,b)$. This is the field of definition of the $G(\bar\BQ_p)$-conjugacy class
$\{\mu\}$. We also set $J_b(\BQ_p):=\{g\in G(\br\BQ_p)\ |\ gb=b\phi(g)\}$; these are the points of a reductive group $J_b$ over $\BQ_p$.

If, in addition, $\mu$ is \emph{minuscule} (i.e. for any root $\alpha$ of $G_{\bar\BQ_p}$, we have $\langle \alpha,\mu\rangle\in \{-1,0,1\}$),
then we say that $(G, [b], \{\mu\})$ are \emph{``local Shimura data''}. (See \cite[\S 5]{RV}.)
 
 \subsubsection{}\label{YandFF}
  
 Let $S=\Spa(R, R^+)$ be an affinoid perfectoid adic space with $R^+$ a $k=k(\br E)$-algebra. Let $\varpi\in R^+$ be a pseudo-uniformizer of $R$,  with Teichm\"uller lift $[\varpi]$ in the ring of Witt vectors $W(R^+)$. Consider   $W(R^+)$ with its $([\varpi], p)$-adic topology and
set
 \[
 \CY_{[0,\infty)}(S):=\Spa(W(R^+), W(R^+))-\{[\varpi]=0\}
 \]
which is an analytic adic space independent of the choice of $\varpi$. We will denote by $\phi$ the isomorphism of $\CY_{[0,\infty)}(S)$ obtained from the Frobenius of  $W(R^+)$.
For a rational number $r\geq 0$, we set
\[
\CY_{[r,\infty)}(S)=\{ |[\varpi]|\leq |p|^r, [\varpi]\neq 0\} \subset  \CY_{[0,\infty)}(S). 
\]
Then $\phi$ maps $\CY_{[r,\infty)}(S)$ to $\CY_{[pr,\infty)}(S)$.
For future reference, we note that we have a natural ring homomorphism
\begin{equation}\label{beta1}
W(R^+)\Big\langle \frac{[\varpi]}{p}\Big\rangle [p^{-1}]\to \Gamma(\CY_{[1,\infty)}(S), \CO) 
\end{equation}
where $W(R^+)\langle[\varpi]/p \rangle$ is the $p$-adic completion of  $W(R^+)[[\varpi]/p]$.

It is useful to also recall the definition of the (adic, relative) Fargues-Fontaine curve as the quotient
\[
X_{FF, S}=\CY_{(0,\infty )}(S)/\phi^\BZ
\]
(\cite[Def. 15.2.6]{SW}, \cite[Def. II.1.15]{FS}). This quotient  is an adic space over $\Spa(W(k)[1/p])$. 

Let $S^\sharp=\Spa(R^\sharp, R^{\sharp, +})\to \Spa(\br E, \CO_{\br E})$ be an untilt of $S$ over $\br E$ which gives a point of the $v$-sheaf $\Spd(\br E)$. The standard canonical surjective ring homomorphism $\theta: W(R^+)\to R^{\sharp+}$ produces   Cartier divisors
$S^\sharp\hookrightarrow \CY(S) $ and also 
$S^\sharp\hookrightarrow X_{FF, S}$. 

 \subsubsection{}\label{ModuliShtuka}
Now choose an open compact subgroup $K\subset G(\BQ_p)$ such that there is a smooth 
 connected affine model $\CG$ over $\BZ_p$ for $G$ with $K=\CG(\BZ_p)$.
 
Let us briefly recall the definition of the moduli space of shtukas $ \Sht_{(G,b,\mu, K)}$ over $\Spd(\br E)$ from \cite[Def. 23.1.1]{SW}. This is the moduli of $\CG$-shtuka ``with one leg bounded by $\mu$ and  framed by $b$''. 
Precisely,
\[
 \Sht_{(G,b,\mu, K)}\to \Spd(\br E)
 \]
  is the presheaf on perfectoid spaces over $k=k(\br E)$ sending the affinoid perfectoid $S=\Spa(R, R^+)$ over  $k$ to the set of isomorphism classes of quadruples $(\CP, S^\sharp, \phi_\CP, \iota_r)$, where
  \begin{itemize}
  \item $\CP$ is a $\CG$-torsor over $\CY_{[0,\infty)}(S)$,
  
  \item $S^\sharp=\Spa(R^\sharp, R^{\sharp, +})\to \Spa(\br E, \CO_{\br E})$ is an untilt of $S$ over $\br E$ so that $(S, S^\sharp)$ is a point of the $v$-sheaf $\Spd(\br E)$,
  
  \item $\phi_\CP$ is a $\CG$-torsor isomorphism 
  \[
  \phi_\CP: (\phi^*\CP)_{|\CY_{[0,\infty)}(S)-S^\sharp}\xrightarrow{\ \sim\ } \CP_{|\CY_{[0,\infty)}(S)-S^\sharp},
  \]
 
  \item $\iota_r$ is an isomorphism of $G$-torsors
  \[
  \iota_r: \CP_{|\CY_{[r,\infty)}(S)}\xrightarrow{\ \sim\ } G\times \CY_{[r,\infty)}(S)
  \]
  for $r\gg 0$, under which $\phi_\CP$ is identified with $b\times \phi$.
   \end{itemize}
   
   We require that $\phi_\CP$ is meromorphic along the Cartier divisor $S^\sharp$ with pole bounded by $\mu$
   at all geometric rank $1$ points. Also we identify $\iota_{r'}$ with $\iota_r$ for $r'>r$, when these are compatible, i.e. when ${\iota_{r}}_{|\CY_{[r',\infty)(S)}}=\iota_{r'}$.
   We will call $\iota_r$ a ``framing" of the shtuka $\CP$.

From  \cite[Prop. 19.5.3]{SW}, it follows that $ \Sht_{(G,b,\mu, K)}$ is a $v$-sheaf over $\Spd(\br E)$. The isomorphism class of  $\Sht_{(G,b,\mu, K)}$ over $\Spd(\br E)$ only depends on the $\phi$-conjugacy class $[b]$. The group $J_b(\BQ_p)$ acts on $\Sht_{(G,b,\mu, K)}$ by composing $\iota_r$ with right multiplication by $j^{-1}\in J_b(\BQ_p)\subset G(\br\BQ_p)=G(W(k)[1/p])$.

Note that, although the definition uses  the group scheme $\CG$, the notation only depends on $K=\CG(\BZ_p)$. This is justified, since as we shall see below, 
 $\Sht_{(G,b,\mu, K)}$  only depends on $K=\CG(\BZ_p)$, and not on the particular choice of  $\CG$.
 
 By the  standard equivalence of categories between vector bundles of rank $n$ and $\GL_n$-torsors (see \cite[Thm. 19.5.2]{SW} for the adic space version),
when $\CG=\GL_n$ over $\BZ_p$, $ \Sht_{(G,b,\mu, K)}$ can be described as the moduli of quadruples $(\sV, S^\sharp, \phi_\sV, \iota_r)$, where $\sV$ is vector bundle of rank $n$ over $\CY_{[0,\infty)}(S)$, and the rest are corresponding data as above. To make the disctinction, we will sometimes call $(\sV,\phi_\sV)$   a \emph{vector space} shtuka.

  \subsubsection{}\label{shtukaViaFF} By \cite[Prop. 23.3.1]{SW}, see also \cite[Prop. 4.6.1]{PRlocal}, the $S=\Spa(R, R^+)$-points of $\Sht_{(G, b,\mu, K)} $ are in bijection to isomorphism classes of quadruples 
 \[
 (S^\sharp, \CE, f, \BP)
 \]
 consisting of
 
 \begin{itemize}

\item an untilt $S^\sharp$ of $S$ over $\br E$,
 
 \item 
 a $G$-torsor $\CE$ over the relative Fargues-Fontaine curve $X_{FF, S}$,
 
 \item a pro-\'etale $\underline{\CG(\BZ_p)}$-torsor $\BP$ over $S$, 
 
 \item 
 a meromorphic isomorphism with pole bounded by $\mu$
 \[
 f: \CE^b _{|X_{FF, S}-S^\sharp}\xrightarrow{\sim } \CE _{|X_{FF, S}-S^\sharp},
 \]
 \end{itemize}
such that $\CE$ is trivial over each geometric point of $S$, and such that the $\underline{\CG(\BZ_p)}$-torsor $\BP$ pushes out to the pro-etale $\underline{G(\BQ_p)}$-torsor
over $S$ that corresponds to $\CE$ (see  \cite[Thm. 22.5.2]{SW}). 

Here, $\CE^b$ is the $G$-torsor over the relative Fargues-Fontaine curve $X_{FF, S}$ obtained from $b$, see \cite[\S 22.4]{SW}.  In fact, $\CE^b$ is ``constant'', i.e. it is obtained by pullback from the absolute base. 

When $\CG=(\GL_n)_{\BZ_p}$, there is, as above, a corresponding ``vector space description'' 
in terms of quadruples $(S^\sharp, \CV, f, \BL)$, with $\CV$ a vector bundle of rank $n$ over $X_{FF,S}$ and $\BL$ a pro-\'etale $\und{\BZ_p}$-local system over $S$.

\subsubsection{} We now recall some additional properties.
As explained in \cite[Prop. 23.3.3, \S 24.1]{SW}, there is an \'etale (\emph{``Grothendieck-Messing"}) \emph{period morphism}
\[
\pi_{\rm GM}: \Sht_{(G,b,\mu, K)}\to {\rm Gr}_{G, \Spd(\br E),\leq \mu} 
\]
of $v$-sheaves over $\Spd(\br E)$. Here, ${\rm Gr}_{G, \Spd(\br E),\leq \mu}$ is the closed $v$-subsheaf of the $B_{\rm dR}$-affine Grassmannian for $G$, see \cite[Lect. 19-20]{SW}. By the above, $\Sht_{(G,b,\mu, K)}$ only depends on $K=\CG(\BZ_p)$, and not on the particular choice of smooth model $\CG$,  see \cite[\S 23.3]{SW}. By \cite[Thm. 23.1.4]{SW}, $\Sht_{(G,b,\mu, K)}$ is a locally spatial diamond.  As in \cite[Prop. 2.25]{Gleason}, we can see that $\Sht_{(G,b,\mu, K)}\to {\rm Spd}(\br E)$ is separated, i.e. the diagonal
\[
\Sht_{(G,b,\mu, K)}\to \Sht_{(G,b,\mu, K)}\times_{{\rm Spd}(\br E)} \Sht_{(G,b,\mu, K)}
\]
is a closed immersion of $v$-sheaves.

For simplicity, when the data $(G,b,\mu)$ are fixed, we will often write $\Sht_K$ instead of $\Sht_{ G,b,\mu, K}$.  For a compact open subgroup $K'\subset K$ we have 
 \[
 \Sht_{K'}\to \Sht_K
 \]
 which is a finite \'etale morphism of diamonds over $\br E$.  Taking an inverse limit gives 
 \[
 \Sht_\infty:=\varprojlim_{K'\subset K}\Sht_{K'}\to \Sht_K.
 \]
This is a pro-\'etale cover with group $K=\CG(\BZ_p)$ and gives a pro-\'etale $\underline{\CG(\BZ_p)}$-torsor $\BP^{\rm univ}$ over $\Sht_K$. This torsor specializes,
at each $S$-point as given above, to the torsor $\BP$ in  the quadruple $(S^\sharp, \CE, f, \BP)$.
Pushing out via $\underline{\CG(\BZ_p)}\subset \underline{G(\BQ_p)}$ gives a pro-\'etale $\underline{G(\BQ_p)}$-torsor $\BP^{\rm univ}_\eta:=\underline{\CG(\BQ_p)}\times_{\underline{\CG(\BZ_p)}}\BP^{\rm univ}$ over $\Sht_K$.

\subsubsection{}  Assume now that $\mu$ is minuscule, i.e. that $(G, [b], \{\mu\})$ are local Shimura data. 

Let $\mathscr{F}_{G, \mu}=G/P_\mu$ be the smooth projective $E$-variety parametrizing the parabolics of $G$ of type  $\mu$; here
\[
P_\mu=\{g\in G \ |\ \lim_{t\to\infty} \mu(t)g\mu(t)^{-1}\ {\rm exists}\}.
\]
Denote by $\br{\mathscr{F}}_{G, \mu}$ its base change to $\br E$. 
Then there is a Bialynicki-Birula style isomorphism of $v$-sheaves over $\Spd(\br E)$
\[
{\rm Gr}_{G, \Spd(\br E),\leq\mu}={\rm Gr}_{G, \Spd(\br E), \mu}\xrightarrow{ \sim\ } \br{\mathscr{F}}_{G, \mu}^\diam.
\]
(See \cite[Prop. 19.4.2]{SW}.) 

Using this together with the \'etaleness of $\pi_{\rm GM}$, Scholze-Weinstein show that  $ \Sht_{(G,b,\mu, K)}$
is  represented by a smooth analytic (adic) space $\CM_{G,b,\mu, K}$ over $\br E$, 
i.e. we have
 \[
 \Sht_{(G,b,\mu, K)}=\CM_{ G,b,\mu, K }^\diam.
 \]
 The adic space $\CM_{G,b,\mu, K}$ is quasi-separated and locally of finite type over $\Spa(\br E)$ and so it can be considered as a rigid analytic variety
 over $\Sp(\br E)$.  
 By definition (\cite[Def. 24.1.3]{SW}), the rigid analytic  variety $\CM_{G,b,\mu, K}$ over $\Sp(\br E)$,
  is the \emph{local Shimura variety} attached to the data $(G, [b], \{\mu\} )$ and $K\subset G(\BQ_p)$.
  When $G=\GL_n$ (and $\mu$ is minuscule), $\CM_{G,b,\mu, K}$ is a \emph{Rapoport-Zink space}, i.e. it is isomorphic to the rigid generic fiber 
  of a Rapoport-Zink formal scheme, see \cite[Thm. 24.2.5]{SW}.

\subsection{$p$-adic period rings}
 
Here we quickly review the, more or less standard, construction of certain versions of Fontaine's $p$-adic period rings for higher dimensional bases. These  appear in several places and, in particular, in \cite{KL}, \cite{Br},  \cite[\S 2, \S 3]{Sh}, and also relate to the constructions in \cite[\S 6]{ScholzeHodge} and \S \ref{YandFF}. For further details, we refer the reader to these articles.

\subsubsection{}\label{sss:covers} Let $R$ be a normal, Noetherian, and flat $\CO_F$-algebra which is  complete for the $p$-adic topology.
 Fix an algebraic closure $\overline{{\rm Frac}(R)}$ of ${\rm Frac}(R)$ containing $\overline F$. Let $\overline R$ be the union of all finite $R$-subalgebras $R'$ of $\overline{{\rm Frac}(R)}$ with the property that $R'[p^{-1}]$
is \'etale over $R[1/p]$ and note that $\overline R$ is a normal domain. Then $\overline R[1/p]$ is the union of all finite \'etale extensions of $R_F=R[1/p]$
in $\overline{{\rm Frac}(R)}$ and we set
\[
\CG_{R}:={\rm Gal}(\overline R[1/p]/R_F).
\]
Finally, we let $\widehat{\overline R}$ be the $p$-adic completion of $\overline R$; then $ \widehat{\overline R}$ is a $p$-torsion free $\CO_C$-algebra, $\overline R\hookrightarrow \widehat{\overline R}$ and $p\widehat{\overline R}\cap \overline R=p\overline R$, see \cite[Prop. 2.0.3]{Br}. The action of $\CG_{R}$ on $\overline R$ extends continuously to $\widehat{\overline R}$. Then $(\widehat{\overline R}[p^{-1}], \widehat{\overline R})$ is a perfectoid Huber pair over $(C,\CO_C)$, see \cite[Lem. 3.7]{Br}, \cite[Lem. 10.1.6]{SW}.\footnote{In \cite{Br, Sh}, $R$ is assumed to satisfy some additional properties, but these are not needed for the statements we consider here.} For brevity, we set 
\[
(A,A^+):=(\widehat{\overline R}[p^{-1}], \widehat{\overline R}).
\]
By construction, the perfectoid affinoid $\Spa(A, A^+)$ is the completion of a pro-\'etale cover of $\Spa(R[p^{-1}], R)$ and 
\begin{equation}\label{Cover}
\Spd(R[p^{-1}], R)=\Spd(A, A^+)/\underline{\CG_{R}},
\end{equation}
 by \cite[Lem. 10.1.7]{SW}.
 
 \subsubsection{}\label{periodrings}
 Let $\BA_{\rm max}(A, A^+)$ be the $p$-adic completion of 
\[
 W(A^{ \flat +})[ ({\rm ker}(\theta))/p]\subset W(A^{ \flat +})[p^{-1}],
\]
 where $(A^\flat, A^{ \flat +})$ is the tilt of $(A, A^+)$ and $\theta: W(A^{ \flat +})\to A^{+}$ is the standard homomorphism. We define
 \[
  \BB^+_{\rm max}(A, A^+):=\BA_{\rm max}(A, A^+)[p^{-1}].
 \]
The homomorphism $\theta$ extends to give $\theta_{\rm max}: \BA_{\rm max}(A, A^+)\to A=A^+[p^{-1}]$. 

We also define
\[
\BB^+_{\rm dR}(A, A^+):=\varprojlim_n W(A^{ \flat +})[p^{-1}]/ ({\rm ker}(\theta))^n.
\]
Choose a compatible system $(\zeta_{p^m})_{m\geq 1}$ of primitive $p^m$-power roots of unity in $\overline F\subset C$,
i.e. with $\zeta_{p^{m+1}}^p=\zeta_{p^m}$, for all $m\geq 1$. Set
\[
\epsilon:=(1, \zeta_p, \zeta_{p^2}, \ldots )\, {\rm mod}\, p\in \CO_{ C^\flat}\subset A^{ \flat +},
\]
and let $[\epsilon]\in W(A^{ \flat +})$ be the Teichm\"uller lift of $\epsilon$. Now set
\[
t=\sum^{+\infty}_{m=1}\frac{(-1)^{m-1}}{m}([\epsilon]-1)^m  \in \BB^+_{\rm dR}(A, A^+).
\]
The element $t$ generates the kernel of the canonical homomorphism
\[
\theta_{\rm dR}: \BB^+_{\rm dR}(A, A^+)\to A,
\]
 cf.
\cite[Prop. 5.1.3]{Br}. Set 
\[
\BB_{\rm dR}(A, A^+):=\BB^+_{\rm dR}(A, A^+)[t^{-1}].
\]
As in \cite[Lem. 2.17]{Sh}, $t$ also converges to give an element of $ \BA_{\rm max}(A, A^+)$, and we have
$\theta_{\rm max}(t)=0$, and $\phi(t)=pt$, where $\phi$ is the Frobenius. Note that the ideal $(t)$ of $\BB^+_{\rm max}(A, A^+)$
is $\phi$-stable.
We also set
\[
\BB_{\rm max}(A, A^+):=\BB^+_{\rm max}(A, A^+)[t^{-1}].
\]

We will also use $\BA_{\crys}(A, A^+)$, which is defined as the $p$-adic completion of the divided power envelope of  $W(A^{ \flat +})$ relative to the
 ideal  ${\rm ker}(\theta)$ compatible with the canonical divided power structure on $pW(A^{ \flat +})$. Define
 \[
 \BB^+_{\crys}(A, A^+):=\BA_{\crys}(A, A^+)[p^{-1}], \quad \BB_{\crys}(A, A^+):= \BB^+_{\crys}(A, A^+)[t^{-1}].
 \]
 The actions of Frobenius $\phi$ and of the Galois group $\CG_{R}$ on $W(A^{\flat+})$ commute and naturally extend to  $\BA_{\rm max}(A, A^+)$, $\BA_{\rm cris}(A, A^+)$, and to $\BB_{\rm max}(A, A^+)$ and $\BB_{\crys}(A, A^+)$.  We have
 \[
 \phi(\BB_{\rm max}(A, A^+))\subset \BB_{\crys}(A, A^+)\subset \BB_{\rm max}(A, A^+).
 \] 
 
 Following \cite{Br, Sh}, we will also use the notations
\[
 {\rm B}_{\rm dR}^\nabla(R):=\BB_{\rm dR}(A, A^+),\quad  {\rm B}_{\crys}^\nabla(R):=\BB_{\crys}(A, A^+),\quad 
 {\rm B}_{\rm max}^\nabla(R):=\BB_{\rm max}(A, A^+).
\]

\subsection{Shtukas and local systems}\label{ss:LS}

In this section, we explain an important connection between (framed) shtukas and \'etale local systems over smooth rigid analytic varieties. 

We will restrict our attention to the case $G=\GL_n$, $\CG=(\GL_n)_{\BZ_p}$; then the reflex field is $\BQ_p$.
Set $V=\BQ_p^n$ so that $G=\GL_n=\GL(V)$.
Then the shtukas are given by suitable vector bundles of rank $n$ with Frobenius structure. 
Suppose that $X$ is a smooth connected rigid analytic variety 
over $\Sp(\br \BQ_p)$ with a morphism 
\[
a: X^\diam \to \Sht_{(\GL_n, b, \mu, K)}
\]
of diamonds over $\Spd(\br \BQ_p)$.
 The pull-back of the universal $\und{\GL_n(\BZ_p)}$-torsor $\BP^{\rm univ}$ to $ X^\diam$ gives a pro-\'etale
$\und{\GL_n(\BZ_p)}$-torsor 
\[
\BP_{X^\diam}=a^*(\BP^{\rm univ})
\]
over the diamond $X^\diam$. As in \cite[\S 4]{Hansen}, we see that $\BP_{X^\diam}$ is canonically isomorphic to the
 ``torsor of trivializations'' of a pro-\'etale $\underline{\BZ_p}$-local system $\BL =\BL_{X^\diam}$ over the diamond $X^\diam$. 
 By   \cite[Prop. 4.3]{Hansen}, the category of pro-\'etale $\underline{\BZ_p}$-local systems
 over the diamond $X^\diam$ is equivalent to the category of pro-\'etale $\underline{\BZ_p}$-local systems over the rigid analytic (adic) variety $X$.  Hence, 
 $\BL $ uniquely corresponds to a pro-\'etale $\underline{\BZ_p}$-local system of rank $n$ over the rigid analytic  variety $X$; we also denote this by $\BL $. We denote by $\BL_\eta$ the corresponding (isogeny) pro-\'etale $\underline{\BQ}_p$-local system. Note also that the category of pro-\'etale $\underline{\BZ_p}$-local systems over the rigid analytic variety $X$, is equivalent to the category of \'etale $\BZ_p$-local systems over $X$, see   \cite[Lem. 9.1.11]{KL}, \cite[Prop. 8.2]{ScholzeHodge}. We will continue denoting by $\BL$ the corresponding \'etale $\BZ_p$-local system over $X$.

\begin{proposition}\label{propdeRham}
\begin{altitemize}
\item[i)] The \'etale $ \BZ_p$-local system $\BL $  over $X$ is de Rham in the sense of \cite{ScholzeHodge}, and has crystalline fibers with corresponding Frobenius isocrystal given by $b$ at all classical points of $X$.

\item[ii)] The \'etale $ \BZ_p$-local system $\BL $  over $X$ is horizontal de Rham, in the sense that the corresponding pair $({\rm D}_{\rm dR}(\BL ), \nabla )$  of bundle with connection associated to $\BL $ by the $p$-adic Riemann-Hilbert correspondence of \cite{LiuZhu}, is isomorphic to the trivial pair $(V\otimes_{\BQ_p}\CO_{X}, 1\otimes d)$.

\item[iii)] The (vector space) shtuka $\sV $ over $X^\diam$  obtained from $a  : X^\diam\to  \Sht_{\GL_n, b, \mu, \GL_n(\BZ_p)}$   is isomorphic to the shtuka $\sV(\BL )$ associated to the de Rham local system $\BL $ by \cite[\S 2.6.1]{PRCJM}.
\end{altitemize}
\end{proposition}

\begin{proof}
i) Let $x: \Spa(F)\to X$ be a classical point, i.e. $F/\br \BQ_p$ is a finite extension. The pull-back of $\BL_\eta$ by $x$ gives a continuous $n$-dimensional $\BQ_p$-representation of $\Gal(\bar F/F)$; the corresponding pro-\'etale $\GL_n(\BQ_p)$-torsor over $\Spa(F)$
 is given by the specialization of 
\[
\Sht_\infty\to \Sht_K\xrightarrow{\pi_{\rm GM}} {\rm Gr}_{\GL_n,\leq \mu}
\]
over the resulting $y:=\pi_{\rm GM}\circ a^\times\circ x: {\rm Spa}(F)\to  {\rm Gr}_{\GL_n,\leq \mu}$; here $\pi_{\rm GM}$ is the period morphism. Then $y$   lies in the admissible locus ${\rm Gr}^{\rm adm}_{\GL_n,\leq \mu}$ of ${\rm Gr}_{\GL_n,\leq \mu}$. Suppose $y$ factors through the ``Schubert cell"  ${\rm Gr}_{\GL_n,\lambda}$ for some $\lambda\leq \mu$. Then, by the definition of the admissible locus, $y$ also lies in ${\rm Gr}^{\rm adm}_{\GL_n,\lambda}$ and hence we also have $b\in B(\GL_n, \lambda^{-1})$ (see, for example, \cite[Prop. 3.5.3]{CarSch}). Then the  Galois representation given by specializing $\BL_\eta$, as above, is  crystalline and in particular de Rham of Hodge-Tate weight $\lambda^{-1}$, and with corresponding Frobenius isocrystal given by $b$. (This is essentially a corollary of  ``weakly admissible $\Rightarrow$ admissible" in the sense of Fontaine. See, for example \cite[Prop. 1.4, Prop. 1.11]{Gleason}.\footnote{Note that \cite{Gleason} uses a different normalization for the coweight.}) The claim that $\BL$ is de Rham (in fact, of constant Hodge-Tate weights given by $\lambda^{-1}$), now follows from the rigidity theorem of \cite{LiuZhu}. 

ii) We  now consider the triple $({\rm D}_{\rm dR}(\BL ), \nabla, {\rm Fil}^{\bullet})$ associated to the de Rham local system $\BL $ by the $p$-adic Riemann-Hilbert correspondence of \cite{LiuZhu}:  

Denote by $(\CV , f , \BL )$ the data, as in \S\ref{shtukaViaFF}, which correspond to the vector space shtuka $(\sV,\phi_\sV)$ over $X$ obtained  from $a $. Here, $\CV $ is a vector bundle of rank $n$ over the relative Fargues-Fontaine curve\footnote{Since there is no actual adic space $X_{FF, X^\diam}$, this ``vector bundle'' makes sense, as usual, only as a section of an appropriate $v$-stack.} $X_{FF, X^\diam}$ and $f $ is a meromorphic isomorphism from $\CV^b$ to $\CV $  which has pole bounded by 
$\mu$ at the universal untilt.  We claim that $f$   provides an isomorphism
\begin{equation}\label{iso1}
V\otimes_{\und{\BQ_p}} \CO \BB_{\rm dR}\xrightarrow{\ \sim\ } \BL\otimes_{\und{\BZ_p}}\CO \BB_{\rm dR}
\end{equation}
of sheaves of $\CO \BB_{\rm dR}$-modules on $X_{\text{pro\'et}}$ (see \cite[Def. 6.8]{ScholzeHodge}, \cite{ScholzeHodgeCor} for the definition of $\BB_{\rm dR}$, $\CO \BB_{\rm dR}$). Consider $ S=\Spa(R, R^+)\to X^\diam$ over $\Spd(\br\BQ_p)$ given by the tilt of the perfectoid affinoid completion $\widehat \CW=S^\sharp=\Spa(R^\sharp, R^{\sharp+})$ of a pro-\'etale $\CW\to X$, as in \cite[Def. 4.3]{ScholzeHodge}.
The pull-back of $\CV$ to $S$ is, by construction \cite[\S 22.3]{SW},
\[
\CV_S=\BL\otimes_{\und{\BZ_p}}\CO_{X_{FF, S}}.
\]
On the other hand, by definition,
\[
\CV^b_S=(\br V\otimes_{\br\BQ_p}\CO_{\CY_{(0,\infty)}(S)})^{\phi=1},
\]
where $\phi$ acts on $\br V=V\otimes_{\BQ_p}\br\BQ_p$ by   $b\otimes {\rm Frob}$. 
It follows that the completion $\widehat{\CV^b_S}$ of $\CV^b_S$ along the divisor of $X_{FF, S}$ given by the untilt $S^\sharp$ is  
\[
\widehat{\CV^b_S}=V\otimes_{\BQ_p}\widehat{\CO}_{\CY_{(0,\infty)}(S), S^\sharp}=V\otimes_{\BQ_p}\BB^+_{\rm dR}(R^\sharp, R^{\sharp+}).
\]
The similar completion of $\CV_S$ is
\[
\widehat{\CV_S}=\BL\otimes_{\BZ_p} \BB_{\rm dR}^+(R^\sharp, R^{\sharp+}).
\]
Hence, by completing $f_S: \CV^b_S\dashrightarrow \CV_S$, we obtain an isomorphism
\[
V\otimes_{\BQ_p} \BB_{\rm dR}(R^\sharp, R^{+\sharp})\xrightarrow{\ \sim\ }  \BL\otimes_{\BZ_p} \BB_{\rm dR}(R^\sharp, R^{\sharp+}).
\]
By \cite[Thm. 6.5]{ScholzeHodge}, $\BB_{\rm dR}(R^\sharp, R^{+\sharp})=\BB_{\rm dR}(\CW)$, the sections over $\CW\to X$ of the sheaf $\BB_{\rm dR}$ on $X_{\proet}$. We can see that, for variable $\CW\to X$, the above isomorphisms give an isomorphism of sheaves of $\BB_{\rm dR}$-modules over $X_{\proet}$. (The required patching can be checked after passing to a cover for the $v$-topology and it follows since $f$ is defined over $X^\diam$.) 
We can now apply $-\otimes_{\BB_{\rm dR}}\CO\BB_{\rm dR}$ and obtain the desired isomorphism (\ref{iso1}) of sheaves.

By the definition in \cite[\S 3.2]{LiuZhu}, we have
\[
{\rm D}_{\rm dR}(\BL)={\rm R}^0\nu_*(\BL\otimes_{\und{\BZ_p}} \CO \BB_{\rm dR})\simeq {\rm R}^0\nu_*(V\otimes_{\und{\BQ_p}} \CO \BB_{\rm dR})
=V\otimes_{\BQ_p} {\rm R}^0\nu_*(\CO \BB_{\rm dR}),
\]
which is $V\otimes_{\BQ_p}\CO_{X}$, with its trivial connection. Here, $\nu: X_{\text{pro\'et}}\to X_{\text{\'et}}$ is the natural morphism
of sites as in \emph{loc. cit.} and we have used (\ref{iso1}). This shows part (ii).

iii) The shtuka $\sV(\BL)$ associated to the deRham local system $\BL$ in \cite[\S 2.6.1]{PRCJM} is constructed, by  the ``glueing''  procedure described in \emph{loc. cit.} Prop. 2.5.1, from $\BL$ and a $\BB_{\rm dR}^+$-submodule sheaf $\BM_0$ of $\BM\otimes_{\BB^+_{\rm dR}}\BB_{\rm dR}$, where $\BM=\BL\otimes_{\und\BZ_p}\BB^+_{\rm dR}$. Furthermore, by \cite[Prop. 2.6.3]{PRCJM}, $\BM_0$ is given by the construction of \cite[Prop. 7.9]{ScholzeHodge} as
\[
\BM_0=({\rm D}_{\rm dR}(\BL)\otimes_{\CO_X}\CO\BB^+_{\rm dR})^{\nabla=0}.
\]
Since ${\rm D}_{\rm dR}(\BL)=V\otimes_{\BQ_p}\CO_{X}$ by the above, $\BM_0=V\otimes_{\und\BQ_p}\BB^+_{\rm dR}$. Note that, by \cite[Thm. 3.9 (iv)]{LiuZhu}, we see that the filtration ${\rm Fil}^{\bullet}$ on the bundle
${\rm D}_{\rm dR}(\BL)=V\otimes_{\BQ_p}\CO_{X}$ (which satisfies the Griffiths transversality condition) 
also gives, by the construction of \cite[Prop. 7.9]{ScholzeHodge}, the $\BB^+_{\rm dR}$-submodule
$\BM=\BL\otimes_{\und{ \BZ_p}} \BB_{\rm dR}^+$ of $\BM_0\otimes_{\BB_{\rm dR}^+}\BB_{\rm dR}$. Part (iii) now follows from the above, the construction of the shtuka $(\sV,\phi_\sV)$ from the data $(\CV,f,\BL)$ which is implicit in the proof of the equivalence of \S\ref{shtukaViaFF}, and the construction of the shtuka $\sV(\BL)$ in \cite[\S 2.6.1]{PRCJM}.
\end{proof}

 \section{Statements and proofs}\label{s:thms}
 
 \subsection{$p$-adic Borel extension}\label{results} Suppose that $F$ is a finite field extension of $\br E$. 
 Let $ \BB=\BB_{F}=\Spa(F\langle t \rangle, \CO_{F}  \langle t \rangle)$ be the closed unit disk over $\Sp(F)$, $\BB^\times=\BB^\times_{F}$ the punctured unit disk over $\Sp(F)$, and let $s\geq 1$, $t\geq 0$ be integers.
The main extension result is:
 
  \begin{theorem}\label{mainThm} 
  Suppose that $(G, [b],\{\mu\})$ are local shtuka data and $K\subset G(\BQ_p)$ is as in \S \ref{ModuliShtuka}. Let  
 $
  a^\times:  (\BB^{\times s}\times \BB^t)^\diam\to {\Sht}_{(G,b,\mu, K)}
 $
  be a morphism of diamonds over $\Spd(\br E)$. Then there is a  unique morphism of diamonds 
$
 a: (\BB^{s+t})^{\diam}= (\BB^{s}\times \BB^t)^\diam\to {\Sht}_{(G,b,\mu, K)}
$ 
 over $\Spd(\br E)$ which extends $a^\times$.
   \end{theorem}

  We will give the proof in \S \ref{mainProof} below. In the rest of the paragraph, we just list some corollaries. Taking $s=1$, $t=0$, we obtain:
  
 \begin{corollary}\label{BorelExt1}
Every morphism $a^\times: \BB^{\times\diam}\to {\Sht}_{(G,b,\mu, K)}$ over $\Spd(\br E)$ extends uniquely to a morphism $a: \BB^{ \diam}\to {\Sht}_{(G,b,\mu, K)}$.
 \hfill $\square$
 \end{corollary}
 
Now let $X$ be a smooth connected affinoid rigid analytic variety over ${\rm Sp}(F)$. For every $F$-valued classical point of $X$, there is a  rational open neighborhood of $X$ which is 
isomorphic to $\BB^t=\BB^t_F$, for some $t\geq 0$. By applying Theorem \ref{mainThm} to $s=1$, we obtain:
 
 \begin{corollary}\label{mainThm0}  Suppose that $(G, [b],\{\mu\})$ are local shtuka data and $K\subset G(\BQ_p)$ is as above. Let  $a^\times: (\BB^\times\times X)^\diam\to {\Sht}_{(G,b,\mu, K)}$ be a morphism over $\Spd(\br E)$ and suppose $z\in \BB\times X$ is a classical point in the complement $(\BB \times X)\setminus (\BB^\times\times X)$. There is a  rational open neighborhood $U_z\subset \BB \times X$ of $z$ and a unique morphism 
 \[
 a_z : U_z^\diam\to {\Sht}_{(G,b,\mu, K)}
 \] over $\Spd(\br E)$
 which extends the restriction of $a^\times$ to $ (U_z\cap (\BB^\times\times X))^\diam$.
  \hfill $\square$
 \end{corollary}

 Suppose now that $(G, [b],\{\mu\})$ are local Shimura data, i.e. $\mu$ is minuscule. Let $\CM_{G,b,\mu, K}$ be the corresponding local Shimura variety over ${\rm Sp}(\br E)$. 
Since $\Sht_{G,b,\mu, K}=\CM_{G,b,\mu, K}^\diam$, the following corollaries follow easily from Theorem \ref{mainThm}, Corollary \ref{mainThm0}, and  the full-faithfulness of the diamond functor on normal analytic adic spaces over $\br E$, see \cite[Prop. 10.2.3]{SW}.

 \begin{corollary}\label{mainLSV}
  Suppose that $(G, [b],\{\mu\})$ are local Shimura data and $K\subset G(\BQ_p)$ is as above. Let  
  $
  a^\times:  \BB^{\times s}\times \BB^t\to {\CM}_{ G,b,\mu, K }
  $
  be a morphism of  rigid analytic varieties  over $ {\rm Sp}(\br E)$. Then there is a  unique morphism of rigid analytic varieties
$
 a: \BB^{s+t }= \BB^{s}\times \BB^t \to {\CM}_{ G,b,\mu, K }
 $
 over $\Sp(\br E)$ which extends $a^\times$.  \hfill $\square$
  \end{corollary}

As above, for a general  smooth connected affinoid rigid analytic variety $X$  over ${\rm Sp}(F)$, we obtain:
 
  \begin{corollary}\label{mainThm2}  Suppose that $(G, [b],\{\mu\})$ are local Shimura data and $K\subset G(\BQ_p)$ is as above. Let  $a^\times:  \BB^\times\times X \to \CM_{G,b,\mu, K}$ be a morphism over $\Spd(\br E)$ and suppose $z\in \BB\times X$ is a classical point in the complement $(\BB \times X)\setminus (\BB^\times\times X)$. There is a  rational open neighborhood $U_z\subset \BB \times X$ of $z$ and a unique morphism 
 \[
 a_z : U_z \to \CM_{G,b,\mu, K}
 \] 
 over $\Sp(\br E)$
 which extends the restriction of $a^\times$ to $ U_z\cap (\BB^\times\times X)$.
  \hfill $\square$
 \end{corollary}
 
 \begin{remark} When 
$(G, [b], \{\mu\})$ are local Shimura data with $G=\GL_n$,  \cite[Thm. 1.1]{OSZP} implies the stronger result that there is an extension $a:\BB\times X\to \CM_{G,b,\mu, K}$ of $a^\times$. Their proof  uses also a global argument which requires an integral model of a global Shimura variety over $\CO_{\br E}$. This argument easily extends to the case that $(G, [b], \{\mu\})$ is of ``local Hodge type''. However, in general,
such an integral model is unavailable.
\end{remark}

\subsection{$p$-adic Brody hyperbolicity}\label{results2} 
Let $(G, [b],\{\mu\})$ be local shtuka data, and $K\subset G(\BQ_p)$ as above. Again, $F$ is a finite field extension of $\br E$.

\begin{theorem}\label{BrodySht}
Every morphism  $Y^\diam\to \Sht_{(G,b,\mu, K)}$ of diamonds over $\Spd(\br E)$, where $Y$ is a smooth irreducible rational curve over $F$, is constant, i.e.  factors through an $F$-rational point
$Y^\diam\to \Spd(F)\to  \Sht_{(G,b,\mu, K)}$ over $\Spd(\br E)$.
\end{theorem}

For local Shimura varieties $ \CM_{G,b,\mu, K}$, this implies:

\begin{corollary}\label{BrodyCor}
 Every rigid analytic morphism $Y^{\rm an}\to \CM_{G,b,\mu, K}$, where $Y$ is a smooth irreducible rational curve over $F$, is constant.  \qed
\end{corollary}  

We give this proof now assuming Theorem \ref{mainThm} which we will show later; in fact, the proof just needs the truth of Corollary \ref{BorelExt1}.

\begin{proof} 
The proof of Theorem \ref{BrodySht} uses 
the following statement which shows that the $p$-adic Borel extension property alone implies this form of $p$-adic Brody hyperbolicity, for reasonably general diamonds. Recall, $Y$ is a smooth irreducible rational curve over $F$.

\begin{proposition}\label{borel-implies-brody-diamond-version}
    Let $\mathcal{D}$ be a locally spatial diamond that is separated over $\Spd(F)$. Suppose that for  every  finite extension $L$ of $F$, every morphism of diamonds $a^\times : (\BB^\times_L)^\diam \rightarrow \mathcal{D}$ over $\Spd(F)$ extends to a (necessarily unique) morphism $a : \BB_L^\diam \rightarrow \mathcal{D}$ over $\Spd(F).$ 
    Then for every finite extension $L$ of $F$,   every  morphism $f : Y_L^\diam \rightarrow \mathcal{D}$ over $\Spd(F)$ is constant, i.e. $f$ factors through a $L$-rational point 
    \[ Y_L^\diam \rightarrow \Spd(L) \rightarrow  \mathcal{D}\]
    over $\Spd(F).$
\end{proposition}
\begin{proof}

We start with a useful lemma.

\begin{lemma}\label{con1}
Let $g: S\to \Spec(F)$ be an algebraic variety  and $\alpha: S^\diam\to \calD$  a morphism of diamonds over $\Spd(F)$. 

a) Suppose there is a variety $T$ with a finite faithfully flat morphism $\pi: T\to S$ such that $\alpha\circ \pi^\diam: T^\diam \to \calD$ factors through a morphism
$s: \Spd(F)\to \calD$ over $\Spd(F)$, i.e. there exists such $s$ with $\alpha\circ \pi^\diam=s\circ (g\circ \pi)^\diam$. Then $\alpha$  factors through  $s: \Spd(F)\to \calD$.

b) Suppose there exists a finite field extension $F'$ of $F$ such that 
\[
\alpha': (S\times_FF')^\diam\to \calD\times_{\Spd(F)}\Spd(F')\to \calD
\]
 factors through a morphism
$s': \Spd(F')\to \calD$ over $\Spd(F)$, i.e. there exists such $s'$ with $\alpha'=s'\circ (g')^\diam$. Then $\alpha$ factors through a morphism $s: \Spd(F)\to \calD$ over $\Spd(F)$, i.e. there exists such $s$  with
$\alpha=s\circ g^\diam$.
\end{lemma}

\begin{proof}
a) We have 
\[
(s\circ g^\diam)\circ \pi^\diam=s\circ (g\circ \pi)^\diam=\alpha\circ \pi^\diam.
\]
 By \cite[Prop. 10.3.7]{SW}, $|T^\diam|=|T|$ and $|S^\diam|=|S|$, where $|T|$ and $|S|$ are the topological spaces underlying the corresponding analytic adic spaces. It follows that $|\pi^\diam|: |T^\diam|\to |S^\diam|$ is surjective.  We can now see, using \cite[Lem. 17.4.9]{SW}, that $\pi^\diam: T^\diam\to S^\diam$ is $v$-surjective, hence this implies $s\circ g^\diam=\alpha$, as desired. 

b) By replacing $F'$ by its Galois hull over $F$ we can assume that $F'/F$ is finite Galois with group $\Gamma$; then $\Spd(F)=\Spd(F')/\underline \Gamma$. We observe that, for all $\gamma\in \Gamma$, we have $s'\circ \gamma^\diam=s'$. Indeed, $(g')^\diam: (S\times_FF')^\diam\to \Spd(F')$ is $v$-surjective and we have 
\[
s'\circ\gamma^\diam\circ (g')^\diam=s' \circ (g'\circ\gamma)^\diam=s'\circ (g')^\diam.
\]
Denote by $q: \Spec(F')\to \Spec(F)$ the natural morphism. By finite \'etale descent, $s'=s\circ q^\diam$, for some $s: \Spd(F)\to \calD$ and we have $\alpha=s\circ g^\diam$.
\end{proof}

We now continue with the proof of the proposition. Note that if $\Spd(L)\to\mathcal{D}$ is a classical point of the diamond $\mathcal{D}$ over $\Spd(F)$, then the image of $|\Spd(L)|$ in $|\mathcal{D}|$
is a closed point of the topological space $|\mathcal{D}|$. In fact, since $\Spd(L)\to \Spd(F)$ is proper and $\mathcal{D}$ is a separated locally spatial diamond over $\Spd(F)$, 
the morphism $\Spd(L)\to \mathcal{D}$ is proper and $\Spd(L)\to \mathcal{D}\times_{\Spd(F)}\Spd(L)$ is a closed immersion of locally spatial diamonds, see \cite[17.4]{SW}.

   We can  see, using Lemma \ref{con1} (b), that it is enough to show that there is a finite extension $L'/L$ such that $Y_{L'}^\diam\to Y_L^\diam\to  \mathcal{D}$ is constant, i.e. factors through a point $\Spd(L')\to  \mathcal{D}$. This allows us to assume that $Y_L$ contains an $L$-rational point. Further, replacing $\mathcal{D}$ by $\mathcal{D}\times_{\Spd(F)} \Spd(L)$, we may assume also that $L = F.$ It is then also enough to show the statement assuming $Y\subset {\BA}^1_F$.

  Now, let $f :Y_F^\diam \rightarrow \mathcal{D}$ be a morphism of diamonds over $\Spd(F).$ 
  
    We first claim that the underlying map of topological spaces $|f| : |Y_F^\diam| \rightarrow |\mathcal{D}|$ is constant.
    Since the classical points of $|Y_F^\diam| = |Y_F^\an|$ are topologically dense in $|Y_F^\an|$, and since, by the above, classical points are closed in $|\mathcal{D}|$ it suffices to prove that all the classical points of $|Y_F^\an|$ have the same image in $|\mathcal{D}|.$
      
    Suppose for a contradiction, that there exist two classical points $x \neq y$ of $|Y_F^\an|$, such that $|f|(x) \neq |f|(y)$ in $|\mathcal{D}|.$ Replacing $L=F$ by a finite extension if necessary, we may assume that both $x$, $y$ are $L$-rational points. Further, replacing again $\mathcal{D}$ by $\mathcal{D}\times_{\Spd(F)} \Spd(L)$, we may assume again that $L = F.$ 
    
   Using our assumption at balls around the punctures, i.e. around the complement of $Y$ in $\BA^1_F$, we extend $f: Y^\diam\to \mathcal{D}$ to a morphism $\hat{f} : (\BA^1_F)^\diam \rightarrow \mathcal{D}$ over $\Spd(F)$.
    Let $g : \BB^\times_{F} \rightarrow (\BA^1_F)^\an$ be an analytic function on $\BB_{F}^\times$ with an essential singularity at $0$.    
    Then, the big Picard theorem of $p$-adic analysis (see for instance \cite[Theorem 2.1]{cherry}) implies that there are sequences of classical points $(x_n)_{n \geq 1} \subseteq g^{-1}(x)$ and $(y_n)_{n \geq 1} \subseteq g^{-1}(y)$ in $\BB_{F}^\times$ with $0 = \lim_n x_n = \lim_n y_n.$ 
    Set $a^\times : (\BB_F^\times)^\diam \rightarrow \mathcal{D}$ to be the composite morphism $a^\times = \hat{f} \circ g^\diam.$ Since $|a^\times|(x_n) = |f|(x)$ and $|a^\times|(y_n) = |f|(y)$, $|a^\times|$ cannot extend continuously to the disk $|\BB_F^\diam| = |\BB_F|$.  
    This contradicts our assumption that an extension $a$ exists. 
    This therefore proves the claim that the underlying map of topological spaces $|f| : |Y_F^\diam| \rightarrow |\mathcal{D}|$ is constant.

    By picking an $F$-rational point of $Y$ and considering its image in $\mathcal{D}$, we thus have that $|f|$ factors through the image $|x|(|\Spd(F)|)$ of an $F$-rational point $x : \Spd(F) \rightarrow \mathcal{D}.$ Since, as above,  $x : \Spd(F) \rightarrow \mathcal{D}$ is a closed immersion, in particular an injection of locally spatial diamonds,
    \cite[Prop. 12.15]{SchDiam} now implies that $f$ factors through $x.$ 
 \end{proof}
The proof of Theorem \ref{BrodySht} now follows by  combining Corollary \ref{BorelExt1}  with Proposition \ref{borel-implies-brody-diamond-version}.
\end{proof}

\begin{remark}\label{alternativeHyperbolicity}
An alternative proof of Theorem \ref{BrodySht} can be obtained as follows. We first reduce to the case that $Y$ contains an $F$-rational point. Then, by applying Corollary \ref{BorelExt1}  at the punctures we extend 
$Y^\diam\to \Sht_{(G,b,\mu, K)}$ to $(\BP^1_F)^\diam\to \Sht_{(G,b,\mu, K)}$. Now Theorem \ref{constant} below, shows that such a morphism 
is necessarily constant.
\end{remark}

\begin{remark}\label{variousHyperbolic} 
We note that the above result, Theorem \ref{BrodySht}, may be seen as proving a certain hyperbolicity property for the moduli spaces of shtukas. We recall that a complex space $S$ is said to be  Brody hyperbolic if there are no non-constant holomorphic maps $\BC \rightarrow S$.  The terminology is less settled in the non-archimedean case, where several related notions of hyperbolicity have been proposed. Over $\BC$, it is not too hard to see that Brody hyperbolic complex manifolds further do not admit any non-constant holomorphic maps from (the analytification of) any connected complex group variety. On the other hand, for a rigid analytic variety over a non-archimedean field $K$, the non-existence of non-constant analytic maps from $(\BA_K^1)^\an$ or even $(\BG_{m,K})^\an$, does not guarantee the constancy of analytic maps from connected group varieties. Javanpeykar-Vezzani call rigid analytic varieties (over an \emph{algebraically closed} non-archimedean field) admitting no non-constant analytic maps from the multiplicative group, respectively from connected group varieties,  \emph{analytically pure}, respectively \emph{analytically Brody hyperbolic} (see \cite[Def. 2.3, Def. 2.7]{JV}). Theorem \ref{BrodySht} implies that the moduli spaces of shtukas do not admit non-constant maps from the multiplicative group over finite extensions of the discretely valued field $\Breve{E}$. As such, it is natural to ask whether the same holds for morphisms defined over the completed algebraic closure $\BC_p$ of $\Breve{E}$. Although it would be reasonable to also expect that, our techniques crucially require that the underlying non-archimedean field is discretely valued. 
Similarly, we also do not know if these moduli spaces are analytically Brody hyperbolic in the terminology of \cite[Def. 2.7]{JV}. We further point the reader to \cite{JV}, \cite{cherry-kobayashi} and the references therein for a discussion on various hyperbolicity notions in the non-archimedean setting.     
\end{remark}

 \subsection{The proof of the extension theorem}\label{mainProof}
We now show Theorem \ref{mainThm}. Recall $F$ is a finite field extension of $\br E$ and $\BB=\BB_F$, $\BB^\times=\BB^\times_F$.

 \begin{proof}
For simplicity, we set
 \[
  U^\times  = \BB^{\times s} \times  \BB^t,\qquad U = \BB^{ s} \times  \BB^t=\BB^{s+t}.
 \]
 
 We first show the uniqueness of the extension:    Note that by \cite[Prop. 10.3.7]{SW}, the topological spaces $|U^{\times\diam} |$, resp. $|U^{\diam} |$, of the diamonds $U^{\times\diam} $, resp. $U^{\diam} $, are naturally homeomorphic with the underlying topological spaces $|U^\times |$ and $|U |$ of the analytic adic spaces $U^\times $ and $U $, respectively.
 We can easily see that $|U^\times|=\cup_{n\geq 1}|\BA^s_n\times \BB^t|$ is an open and dense subset of $|U|$.
  Hence, the same is true for 
 $|U^{\times\diam} |$ in $|U^{\diam}|$.
The uniqueness of the extension $a $ of $a ^\times $ in the statement of the theorem follows from this density combined with  \cite[Lem. 17.4.1]{SW} (or \cite[Prop. 12.15]{SchDiam}), since ${\Sht}_{(G,b,\mu, K)}$ is separated over $\Spd(\br E)$.

As a preliminary reduction, we  observe that it is enough to prove that the base change $a^\times_{ F'}$ of $a^\times $ to some finite extension $F'$ of $F$ extends to 
\[
a (F'): U^\diam \times_{\Spd(F)}\Spd(F')\to {\Sht}_{(G,b,\mu, K)}\times_{\Spd(\br E)}\Spd(F') 
\]
over $\Spd(F')$:
Indeed, by replacing $F'$ by its Galois hull over $F$, we can assume that $F'/F$ is Galois with finite Galois group $\Gamma$. Then by   uniqueness as above, we have $\gamma\cdot a (F')=a (F')$, for all $\gamma\in \Gamma$, since this property is true for $a^\times_{ F'}$. Since $\Spd(F)=\Spd(F')/\underline {\Gamma}$,  we deduce the existence of $a $ that extends $a^\times $ by \'etale descent (\cite[Lem. 9.2]{SchDiam}).
\smallskip
 
 {\bf A) Reduction to the case of $\GL_n$:} We now reduce the proof to the case that the group is $\GL_n$: Since $\CG$ is smooth and affine over $\BZ_p$, there exists a closed immersion of group schemes $\rho: \CG\hookrightarrow \GL_n$ over $\BZ_p$;
   set $G'=\GL_n$, $b' =\rho(b)$, $\mu'=\rho\circ \mu$, and take $K'=\GL_n(\BZ_p)$. Then $(G', [b'], \{\mu'\})$ gives local shtuka data with reflex field $E'=\BQ_p\subset E$, and $\rho$ induces a morphism of diamonds 
   \[
  \rho_*: {\Sht}_{(G,b,\mu, K)}\to {\Sht}_{(G',b',\mu', K')}\times_{\Spd(\br \BQ_p)}\Spd(\br E)
   \]
 over $\Spd(\br E)$; this morphism is a closed immersion. (This follows from the argument in the proof of \cite[Prop. 2.25]{Gleason}. Alternatively, it can be deduced  from the \'etaleness of the period morphism together with the fact that  
 \[
 \rho^{\rm Gr}_*: {\rm Gr}_{G,\Spd( E),\leq\mu}\hookrightarrow  {\rm Gr}_{\GL_n, \Spd(\BQ_p), \leq\mu'}\times_{\Spd(\BQ_p)}\Spd(E)
 \]
 is a closed immersion, for example, see \cite[Lem. 19.1.5, Prop. 19.2.3, Lect. 20]{SW}). 
 
 Consider the compositions 
 \[
\rho_*\circ a^\times : U^{\times\diam} \to {\Sht}_{(G',b',\mu', K')}\times_{\Spd(\br \BQ_p)}\Spd(\br E),
 \]
 \[
c^\times: U^{\times\diam} \xrightarrow{ \rho_*\circ a^\times } {\Sht}_{(G',b',\mu', K')}\times_{\Spd(\br \BQ_p)}\Spd(\br E)\to {\Sht}_{(G',b',\mu', K')}.
 \]
 Assuming that the conclusion of Theorem \ref{mainThm} holds for $\GL_n$, the morphism $c^\times$ extends to 
 \[
 c: U^\diam \to {\Sht}_{(G',b',\mu', K')}
 \]
 over $\Spd(\br \BQ_p)$. This gives 
\[
a': U^{\diam} \xrightarrow{ \ } {\Sht}_{(G',b',\mu', K')}\times_{\Spd(\br \BQ_p)}\Spd(\br E)
\]
 over $\Spd(\br E)$ such that $a' =\rho_*\circ a^\times $. 
 Recall that $|U^{\times\diam} |$ is open and dense in $|U^\diam |$ and $\rho_*$ is a closed immersion. 
Since 
$
a'(|U^{\times\diam} |)=(\rho_*\circ a^\times )(|U^{\times\diam} |)\subset \rho_*(|{\Sht}_{(G,b,\mu, K)}|)$,
we see that $a'(|U^{\diam} |)\subset \rho_*(|{\Sht}_{(G,b,\mu, K)}|)$. Hence, by \cite[Lem. 17.4.1 (3)]{SW} (or \cite[Prop. 12.15]{SchDiam}),
$a'$ factors through $\rho_*$, i.e. there is
\[
a : U^\diam \to {\Sht}_{(G,b,\mu, K)}
\]
 with $\rho_*\circ a =a'$. We can now easily see that $a $ extends $a^\times$. Hence, this completes our reduction of the proof to the case $G=\GL_n$. 
\smallskip
 
{\bf B) The case of $\GL_n$:} From here and on, we assume that  $G=\GL_n$, $\CG=(\GL_n)_{\BZ_p}$; then the reflex field is $\BQ_p$.
Set $V=\BQ_p^n$ so that $G=\GL_n=\GL(V)$.
\smallskip

 {\bf B1) Extending the shtuka:} The pull-back of the universal $\und{\GL_n(\BZ_p)}$-torsor $\BP^{\rm univ}$ to $ U^{\times \diam}$ gives a pro-\'etale
$\und{\GL_n(\BZ_p)}$-torsor 
\[
\BP_{U^{\times \diam}}=(a^\times)^*\BP^{\rm univ}
\]
over the diamond $U^{\times \diam}$. As in \S \ref{ss:LS}, $\BP_{U^{\times \diam}}$ is canonically isomorphic to the
 ``torsor of trivializations'' of a pro-\'etale $\underline{\BZ_p}$-local system $\BL^\times$ over the diamond $U^{\times \diam}$.
This uniquely corresponds to an \'etale $\BZ_p$-local system over $U^\times$ which we will also denote by $\BL^\times$.

 By  Prop. \ref{propdeRham} (ii), the local system $\BL^\times $ over $U^\times $, is horizontal de Rham.    We now claim that $\BL^\times $ extends to a de Rham $\BZ_p$-local system $\BL $ over $U $, with the same Hodge-Tate weights. We first note the following general lemma.

\begin{lemma}\label{lemmaExtNew}
	Let $(\CV,\nabla)$ be an analytic vector bundle with integrable connection on $U$.
	 If the restriction $(\CV,\nabla)\vert_{U^\times}$ admits an analytic trivialization over $U^\times$ then $(\CV,\nabla)$ admits a trivialization over $U$. 
\end{lemma}
\begin{proof}
	Since every analytic vector bundle on $U$ is trivial (see \cite[Theorem 1]{lu}), we may pick a trivialization of the underlying vector bundle $\CV \cong \CO_U^r$. By hypothesis, there is a basis $\{\vec{f_1},\ldots,\vec{f_r}\}$ of nowhere vanishing flat sections $\vec{f_i} \in \CO_U^r({U^\times}).$ By \cite[Theorem 9.7]{Sh}, there is some $\epsilon \in |F^\times|$, $0<\epsilon <1$, such that $(\CV,\nabla)\vert_{\BB(\epsilon)^{s+t}}$ admits a basis $\{\vec{g_1},\ldots,\vec{g_r}\}$ of flat sections, with $\vec{g_i} \in \CO_U^r(\BB(\epsilon)^{s+t}).$ Here, $\BB(\epsilon)$ denotes the analytic closed unit disk of radius $\epsilon$. There is then an invertible matrix $M \in \GL_r(F)$, such that $(\vec{f_1},\ldots,\vec{f_r})\vert_{U^\times \cap \BB(\epsilon)^{s+t}} = (\vec{g_1},\ldots,\vec{g_r})\vert_{U^\times \cap \BB(\epsilon)^{s+t}} \cdot M$. This implies that each $\vec{f_i} \in \CO_U^r(U^\times)\cap \CO_U^r(\BB(\epsilon)^{s+t})$ and that $1/\det(\vec{f_1},\ldots,\vec{f_r}) \in \CO_U(U^\times)\cap \CO_U(\BB(\epsilon)^{s+t})$. It follows from \cite[Lemma 2.6]{OSZP} that $\CO_U(U^\times)\cap \CO_U(\BB(\epsilon)^{s+t}) = \CO_U(U)$. Hence each $\vec{f_i} \in \CO_U^r(U)$ extends to a global flat section, and furthermore $\vec{f_1},\ldots,\vec{f_r}$ is a basis of global sections of $\CV$. 
	The horizontal morphism $(\CO_U^r,d) \rightarrow (\CV,\nabla)$ sending the standard basis to $\vec{f_i} \in \CV(U)$ is therefore an isomorphism of flat bundles.
\end{proof}

We shall now apply the argument in the proof of \cite[Thm 5.7]{OSZP}    (which uses crucially the $p$-adic log Riemann-Hilbert correspondence of \cite{DLLZ}) to obtain the required extension $\BL$ of $\BL^\times$.  

The proof of  \cite[Thm 5.7]{OSZP} is written for the case where the boundary divisor $Z$
 is smooth. However, we can apply the argument of \emph{loc. cit.} to first see that $\BL^\times$ extends to a de Rham \'etale $\BZ_p$-local system $\BL_1^\times$ over $U_1^\times := {(\BB^\times)}^{s-1} \times \BB^{t+1}$.  Since the Kummer \'etale extensions to $U_{\text{k\'et}}$ (see \cite[Cor. 6.3.4]{DLLZ-I}) of $\BL_1^\times$ and $\BL^\times$ agree,   $(D_{\text{dR,}\log}(\BL_1^\times),\nabla_{\BL_1^\times})$ agrees with $(D_{\text{dR}, \log}(\BL^\times),\nabla_{\BL^\times})$. We may thus iterate the above argument, inducting on $s$, to obtain the desired extension $\BL$ over $U$.
Moreveover, it follows from  Lemma \ref{lemmaExtNew} that $\BL $ is horizontal de Rham, with $({\rm D}_{\rm dR}(\BL ), \nabla)$ being the (trivial) canonical extension of the trivial pair $({\rm D}_{\rm dR}(\BL^\times), \nabla^\times)$.

The de Rham local system $\BL $ over $U $ gives, by the construction in \cite[\S 2.6]{PRCJM}, a shtuka $\sV(\BL )$ over $U^\diam$; this has one leg bounded by $\mu$ and by 
Proposition \ref{propdeRham} (iii) restricts to the shtuka over $U^{\times\diam} $ obtained from $a^\times$. 
\medskip

{\bf B2) Constructing a framing:} We will now show that the shtuka $\sV(\BL )$ over $U^\diam$ constructed above, admits a   framing by $b$, after possibly enlarging the finite extension $F/\br E$. We first show:

\begin{lemma}\label{crispoints}
There is a finite field extension $F'/F$ such that the base change $\BL_{F'}$ over $U_{F'}$ of the local system $\BL$ is crystalline at all classical points of $U_{F'}$.
\end{lemma}

\begin{proof} We set $\BT=\BT_F=\Spa(F\langle t, t^{-1}\rangle, \CO_F\langle t, t^{-1}\rangle)$.  
Note that we can cover $U=\BB^{s+t}$ by a finite 
union of  open rational subdomains $W\subset U$, each isomorphic to $\BT^{s+t}$. 
For example, $\BB$ is covered by $|x|=1$ and $|x-1|=1$. Using this, we see that it is enough to show that all the restrictions $\BL_{|W}$ satisfy the conclusion of the claim, i.e. for each such $W\simeq \BT^{s+t}\subset U$, there is a finite $F'/F$ such that the base change $\BL_{|W_{F'}}$ is crystalline at all classical points of $W_{F'}$. 
Note that 
\[
W\simeq \BT^{s+t}=\BT^{s+t}_F=\Spa(R[p^{-1}], R)
\]
where $R=\CO_F\langle  t^{\pm 1}_1,\ldots, t^{\pm 1}_{s+t}\rangle$ satisfies the assumptions of \cite[Set-up 3.1]{Sh} and the assumption (BR) (``bonne reduction'') of \cite[3.3]{Sh}, see also \cite{Br}.

We will now apply 
results of Shimizu \cite{Sh}. Note  that the ring ${\rm B}^\nabla_{\rm dR}(R)={\BB}_{\rm dR}(A,A^+)$ agrees with the sections of the pro-\'etale sheaf $\BB_{\rm dR}$ of \cite{ScholzeHodge}   on the corresponding pro-\'etale cover of
$\Spa(R[p^{-1}], R)$ (see \cite[\S 8.2]{Sh}). Also, by   \cite[Lem. 8.7]{Sh} the notion of ``de Rham local system" of \cite{ScholzeHodge}, \cite{LiuZhu}, used above, agrees with the notion used in \cite{Sh}. Similarly, by the above, the representation   $\CG_{R} \to \GL(\Lambda)\subset \GL(V)$ ($\Lambda\subset V$ is a $\BZ_p$-lattice) which corresponds to $\BL_{|W}$
is horizontally de Rham, also in the sense of \cite{Sh}. Hence, since by Proposition \ref{propdeRham} (i), $\BL_{|W}$ has some crystalline fibers, the $p$-adic monodromy result  
\cite[Thm. 7.4]{Sh}, implies that there is a finite extension $F'/F$, such that  the pull-back $\BL_{|W_{F'}}$ of $\BL$ to $W_{F'}\simeq \BT^{s+t}_{F'}$ is ``horizontal crystalline", i.e. the representation   $\CG_{R_{F'}} \to \GL(\Lambda)\subset \GL(V)$ which corresponds to $\BL_{|W_{F'}}$ is horizontal crystalline in the sense of \emph{loc. cit.}.  This implies, in particular, that $\BL_{|W_{F'}}$
is crystalline at all classical points of $W_{F'}$. 
\end{proof}

By the ``preliminary reduction'' in the beginning of the proof, it is enough to show the extension of $a^\times$ exists after we base change to $F'$. Hence, we can replace $F'$ by $F$ and assume that $F'=F$ in the conclusion of Lemma \ref{crispoints}; we can also assume that $F$ is Galois over $\br\BQ_p=W(k)[1/p]$. Recall, $U=\BB^{s+t}=\BB_F^{s+t}$; this is the generic fiber of the smooth $p$-adic formal scheme 
\[
\fkU={\rm Spf}(\CO_F\langle x_1,\ldots, x_{s+t}\rangle)
\]
over $\CO_F$. Lemma \ref{crispoints} with $F'=F$ implies that the local system $\BL$ over $U$ satisfies the``pointwise crystallinity criterion'' of Guo-Yang \cite[Thm. 7.2]{GY} and, hence, is crystalline (with respect to $\fkU$). Recall that, as in \cite[Def. 4.4]{GY}, this means that there is a locally free Frobenius isocrystal $(\CE,\phi_\CE)\in {\rm Isoc}^\phi(\fkU_{k, {\crys}})$, together with a Frobenius equivariant isomorphism of $\BB_{\crys}$-vector bundles
\[
\vartheta: \BB_{\crys}(\CE)\xrightarrow{\sim} \BB_{\crys}\otimes_{\BZ_p} \BL
\]
over $U_{\proet}$.
Here, the notations are as in \emph{loc. cit.}. In particular, ${\rm Isoc}^\phi(\fkU_{k, {\crys}})$ is the category of Frobenius isocrystals over the crystalline site $(\fkU_k/W(k))_\crys$, where $\fkU_k$ is the special fiber $\Spec(k[x_1,\ldots ,x_{s+t}])=\BA^{s+t}_k$.
Moreover, in the above, we can take $(\CE,\phi_\CE)$ to be  the value $\CE_{\crys, \BL}$ of the crystalline Riemann-Hilbert functor of \cite[Thm. 5.20]{GY} at the local system $\BL$, see 
\cite[Thm. 5.20 (v)]{GY} and its proof.

We next claim that $\CE_{\crys, \BL}$ is ``constant'', i.e. obtained by base change from the Frobenius isocrystal  over $\Spec(k)$ which is given by $b$.

Note that, in our set-up, we can consider 
\[
\fkU_0=\Spf(W(k)\langle x_1,\ldots, x_{s+t}\rangle)
\]
with its standard Frobenius lift $\phi_0$ taking $x_i$ to $x_i^p$. Then, by a classical result (see e.g. \cite[Prop. 2.11]{GY}) the Frobenius isocrystal $\CE_{\crys, \BL}$ is functorially uniquely determined by a triple $(\CV_0, \nabla_0, \phi_{\CV_0})$, where
\begin{itemize}
\item $\CV_0$ is a vector bundle  of rank $n$ over $U_0=\fkU_{0,\eta}=\BB^{s+t}_{W(k)[1/p]}$,
\item $\nabla_0$ is a flat connection on $\CV_0$,
\item   $\phi_{\CV_0}$ is a  $\phi_0$-linear Frobenius endomorphism $\phi_{\CV_0}$ whose linearization $\phi^\#_{\CV_0}: \phi^*_0\CV_0\xrightarrow{\ } \CV_0$ is a horizontal    isomorphism. 
\end{itemize}

To prove our claim, it is enough to show that the triple $(\CV_0, \nabla_0, \phi_{\CV_0})$ is constant, i.e. isomorphic to 
$(\br V\otimes_{\br \BQ_p}\CO_{U_0}, 1\otimes d, b\otimes\phi_0)$. 

The vector bundle $\CV_0$ over $U_0$ is trivial, by a result of L\"{u}tkebohmert \cite{lu}. After choosing a basis of $\CV_0$ we can  write the connection $\nabla_0$ via an $n\times n$ matrix $\Omega=(\omega_{ij})$ of $1$-forms $\omega_{ij}\in \Omega^1_{U_0}(U_0)$, while $\phi_{\CV_0}=C\cdot \phi_0 $ for an $n\times n$ matrix $C=(c_{ij})$, $c_{ij}\in \CO_{U_0}(U_0)=\br\BQ_p\langle x_1,\ldots, x_{s+t}\rangle$. By the construction of the crystalline Riemann-Hilbert functor $ \CE_{\crys, -}$ and \cite[Thm. 4.12 (iii)]{GY}, the base change of $(\CV_0, \nabla_0)$ from $\br\BQ_p$ to $F$ is isomorphic to the pair $(\BD_{\rm dR}(\BL), \nabla)$ over $U$ which is
attached to $\BL$ by the $p$-adic Riemann-Hilbert correspondence of \cite{LiuZhu}; as above, the pair $(\BD_{\rm dR}(\BL), \nabla)$ is isomorphic to the trivial pair $(V\otimes_{\BQ_p}\CO_{U}, 1\otimes d)$. Hence, there is an invertible $n\times n$ matrix $g=g(\underline x)$ with entries in 
$\CO_{U}(U)=\CO_{U_0}(U_0)\otimes_{\br\BQ_p}F$, such that
\[
dg\, g^{-1}-g\Omega g^{-1}=0.
\]
This gives $\Omega=g^{-1}dg$ and since for each $\sigma\in \Gamma=\Gal(F/\br\BQ_p)$, $\sigma\Omega=\Omega$, we obtain $\sigma(g)^{-1}d(\sigma(g))=g^{-1}dg$. In turn, this implies $d(\sigma(g)g^{-1})=0$. It follows that $\sigma(g)g^{-1}=:c_\sigma\in F^*$
and so $\sigma\mapsto c_\sigma$ gives a $1$-cocycle $\Gamma\to F^*$. By Hilbert $90$, there is $\delta\in F^*$ such that $c_\sigma=\sigma(\delta)\delta^{-1}$. Then $h:=\delta^{-1}g$ satisfies $\sigma(h)=h$ and $dh\, h^{-1}-h\Omega h^{-1}=0$. Hence, $h$ lies in  $\GL_n(\CO_{U_0}(U_0))$ and gives an isomorphism of $(\CV_0, \nabla_0)$ with the trivial pair $(V\otimes_{\BQ_p}\CO_{U_0}, 1\otimes d)$ over $U_0$. 
 Under the identification $\CV_0=V\otimes_{\BQ_p}\CO_{U_0}=\br V\otimes_{\br\BQ_p}\CO_{U_0}$  given by this isomorphism, the linearization $\phi^\#_{\CV_0}$ of $\phi_{\CV_0}$ is given by  $C\in {\rm End}_{\br\BQ_p}(\br V)\otimes_{\br\BQ_p}\CO_{U_0}(U_0)$. Then $\phi_{\CV_0}=C\cdot \phi_0$. Since $\phi^\#_{\CV_0}$ is horizontal for the connection $\nabla_0=1\otimes d$, we see that $dC=0$, i.e. $C$ is constant, therefore $C=b\in \GL(\br V)=\GL_n(\br\BQ_p)$. 
Hence, the triple $(\CV_0, \nabla_0, \phi_{\CV_0})$ is indeed   isomorphic to 
$(\br V\otimes_{\br\BQ_p}\CO_{U_0}, 1\otimes d, b\otimes\phi_0)$, as claimed. Therefore, $\CE_{\crys, \BL}$ is  constant  and obtained by pull-back from the Frobenius isocrystal over $\Spec(k)$ given by $b$.

It then follows, using the  functoriality  of the construction of  $\BB_{\crys}(\CE)$ in \cite[4.1]{GY}, that we have a Frobenius equivariant isomorphism 
of $\BB_{\crys}$-vector bundles
\[
 \BB_{\crys}(\CE_{\crys, \BL})\simeq \br V\otimes_{\br\BQ_p}\BB_{\crys}
\]
where the Frobenius action on $\br V=V\otimes_{\BQ_p}\br \BQ_p$ is given by $b\otimes {\rm Frob}$. Hence, $\vartheta$ above amounts to a Frobenius equivariant isomorphism
\begin{equation}\label{isoN2}
 \vartheta: \br V\otimes_{\br\BQ_p}\BB_{\crys}\xrightarrow{\sim}   \BL\otimes_{\BZ_p}\BB_{\crys}
\end{equation}
of $\BB_{\crys}$-vector bundles over $U_{\proet}$. We will now use this isomorphism to obtain a $b$-framing of the shtuka $\sV(\BL)$ over $U^\diam$.
We set, for simplicity, $R=\CO_F\langle x_1,\ldots, x_{s+t}\rangle$, so that $U=\Spa(R[p^{-1}], R)$ and consider 
the perfectoid cover $\Spa(A, A^+)\to \Spa(R[p^{-1}], R)$ obtained as in \S \ref{sss:covers}.
Evaluating the isomorphism (\ref{isoN2}) at the corresponding pro-\'etale cover of $U=\Spa(R[p^{-1}], R)$ gives
\begin{equation}\label{iso2}
\br V\otimes_{\br\BQ_p}{\BB}_{\crys}(A, A^+)\xrightarrow{\ \sim\ } \Lambda\otimes_{\BZ_p} {\BB}_{\crys}(A, A^+).
\end{equation}
Base changing (\ref{iso2}) by  ${\BB}_{\crys}(A, A^+)\hookrightarrow {\BB}_{\rm max}(A, A^+)$ gives
\begin{equation}\label{iso3}
\br V\otimes_{\br\BQ_p}{\BB}_{\rm max}(A, A^+)\xrightarrow{\ \sim\ } \Lambda\otimes_{\BZ_p} {\BB}_{\rm max}(A, A^+).
\end{equation}
These isomorphisms respect the actions of Frobenius $\phi$ and $\CG_{R}$: For (\ref{iso3}) the Galois group $\CG_{R}$   acts on the LHS  only on  ${\BB}_{\rm max}(A,A^+)$ and diagonally on the RHS, while   $\phi$ acts diagonally on the LHS and on the RHS only on ${\BB}_{\rm max}(A, A^+)$.

Set $S=\Spa(A^\flat, A^{\flat+})$ for the tilt of $\Spa(A, A^{+})$; this gives a  cover $q: S=\Spa(A^\flat, A^{\flat+})\to U^\diam$. 
Observe that 
\[
{\BB}^+_{\rm max}(A,A^+)=W(A^{\flat+})\Big\langle \frac{p-[p^\flat]}{p}\Big\rangle[p^{-1}]=W(A^{\flat+})\Big\langle \frac{[p^\flat]}{p}\Big\rangle[p^{-1}], 
\]
where $p^\flat=(p, p^{1/p}, \ldots )\,{\rm mod}\, p\in \CO_{C^\flat}\subset A^{\flat+}$ is a pseudo-uniformizer of $A^\flat$.  
Indeed,  $\ker(\theta)\subset W(A^{\flat+})$ is generated by $\xi:=p-[p^\flat]$ (\cite[Lem. 3.6.3]{KL}), and $|\xi|=|p-[p^\flat]|\leq |p|$ is equivalent to $|[p^\flat]|\leq |p|$. Hence, (\ref{beta1}) gives a natural Frobenius equivariant map  
 \[
{\beta}: {\BB}^+_{\rm max}(A,A^+)\to \Gamma(\CY_{[1,\infty)}(S), \CO).
 \]
This also gives
\[
\beta[t^{-1}]: {\BB}_{\rm max}(A,A^+) \to \Gamma(\CY_{[1,\infty)}(S), \CO)[\beta(t)^{-1}].
\]
 Recall that the completion of 
${\BB}^+_{\rm max}(A,A^+)$ along $\ker(\theta_{\rm max})$ identifies with ${\BB}^+_{\rm dR}(A,A^+)$.  
The base changes of  the  ${\BB}^+_{\rm max}(A,A^+)$-modules $\br V\otimes_{\br\BQ_p}{\BB}^+_{\rm max}(A, A^+)$ and $\Lambda\otimes_{\BZ_p} {\BB}^+_{\rm max}(A,A^+)$ by   ${\BB}^+_{\rm max}(A,A^+)\to {\BB}^+_{\rm dR}(A,A^+)$ give
${\BB}^+_{\rm dR}(A,A^+)$-lattices $\BM_0(A,A^+)$ and $\BM(A,A^+)$, respectively, in
\begin{equation}\label{iso4}
 \br V\otimes_{\br\BQ_p}{\BB}_{\rm dR}(A, A^+)\simeq \Lambda\otimes_{\BZ_p} {\BB}_{\rm dR}(A, A^+),
\end{equation}
where the isomorphism is the base change of (\ref{iso3}) by ${\BB}^+_{\rm max}(A,A^+)\to {\BB}^+_{\rm dR}(A,A^+)$. 
Base changing (\ref{iso3}) by $\beta[t^{-1}]$  gives a $\phi$-equivariant meromorphic isomorphism 
 \[
 \br V\otimes_{\br\BQ_p} \CO_{\CY_{[1,\infty)}(S)\setminus \cup_{n\geq 0} \phi^n(S^\sharp)} \xrightarrow{\ \sim\ } \Lambda\otimes_{\BZ_p}  \CO_{\CY_{[1,\infty)}(S)\setminus\cup_{n\geq 0} \phi^n(S^\sharp)}.
 \]
We can see, using the construction of the shtuka $\sV(\BL)$, that this extends to an isomorphism
\[
\iota_{r}: \br V\otimes_{\br\BQ_p} \CO_{\CY_{[r,\infty)}(S)}\xrightarrow{\ \sim\ } \sV(\BL)_{|\CY_{[r,\infty)} (S)}
\] 
 for $r\gg 0$, which is Frobenius equivariant, when the Frobenius on the LHS given by $b\otimes {\rm Frob}$. Indeed, 
 the claim
follows from the construction of $\sV(\BL)$ in \cite[\S 2.5.1, \S 2.6]{PRCJM},
 since the shtuka $\sV(\BL)$ over $U^\diam$ which is associated to 
the de Rham local system $\BL$ over $U$, is given using the   lattices $\BM_0(A,A^+)$ and $\BM(A,A^+)$ and the cover $q$,
see also \cite[Def. 2.6.4, Prop. 2.6.3]{PRCJM}.

 The isomorphism $\iota_r$ gives a $b$-framing of  $\sV(\BL)$ over $S$. We now show that $\iota_r$ descends along the cover $q$ to $U^\diam$. Since $U^\diam=\Spd(A, A^+)/\und{\CG_{R}}$ by \cite[Lem. 10.1.7]{SW}, it is enough to show that $\iota_{r}$ is equivariant for the action of the Galois group $\CG_{R}$. This follows from the construction and the fact that the isomorphism (\ref{iso2}) is Galois equivariant. Hence, we obtain a $b$-framing of $\sV(\BL)$ over $U^\diam$, which we will also denote by $\iota_r$.
\medskip

{\bf B3) Extending the framing:} 

It remains to show that we can adjust the framing $\iota_r$ of $\sV(\BL)$ we just constructed in the section above, so that it extends the original framing $\iota^\times_r$ over $U^{\times\diam}$ given by $a^\times$. 

The discrepancy between the 
restriction $(\iota_r){|_{U^{\times\diam}}}$ and $\iota^\times_r$ is given by a $U^{\times\diam}$-section of the automorphism $v$-sheaf $\tilde G_b:=\und{\rm Aut}(\CE^b)$, where $\CE^b$ is the bundle corresponding to $b$ over the (absolute) Fargues-Fontaine curve. 

 Let $V$ be an analytic adic space over $\Spa(F,\CO_F)$ with $|V|$ connected which admits an open cover
\[
V=\bigcup_{n\geq 1}\Spa(R_n[1/p], R_n).
\]

We will assume that, for each $n$, $R_n$ is a Noetherian, normal and flat $p$-adically complete $\CO_F$-algebra and that the localizations of $R_n$ at height $1$ primes over $p$ are dvrs which are formally smooth over $\CO_F$.

\begin{proposition}\label{BC} Under the above assumptions, we have
$
\tilde G_b(V^\diam)= J_b(\BQ_p).
$
\end{proposition}

Note that  we can write
\[
U^{\times}=\bigcup_{n\geq 1}\Spa(R_n[1/p], R_n)
\]
with 
$
R_n=\CO_F\langle x_1, p^n/x_1, \ldots, x_s, p^n/x_s, x_{s+1}, \ldots  , x_{s+t}\rangle
$
satisfying the above. The topological space $|U^\times|$ is connected. Hence,  Proposition \ref{BC} applies to $V=U^{\times}$ and we obtain
$
\tilde G_b( U^{\times\diam})= J_b(\BQ_p).
$

\begin{proof}
By \cite[Prop.
III.5.1]{FS}, the group sheaf $\tilde G_b$ is a semi-direct product  
\begin{equation}\label{semidirect}
 \tilde G_b=\CU\rtimes \und{J_b(\BQ_p)},
\end{equation}
where  $\CU$ is a unipotent group sheaf which is a successive extension of {\sl positive} (absolute) Banach-Colmez spaces ${\mathcal BC}(\lambda)$,    cf. \cite[Prop. 4.2.5]{PRCJM}.

Set $A^+_n=\widehat{\overline R}_n$, $A_n=\widehat{\overline R}_n[p^{-1}]$, constructed as in \S \ref{sss:covers}. Then
\begin{equation}\label{descent1}
\Spa(R_n[1/p], R_n)^\diam=\Spd(A_n, A^+_n)/\underline{\CG_{R_{n}}},
\end{equation}
with $\CG_{R_n}=\Gal({\overline R}_n[p^{-1}]/R_n[1/p])$.

\begin{lemma}\label{Ginv}  Let $R_n$ be an $\CO_F$-algebra satisfying the above assumptions. 
\begin{itemize}
\item[i)] We have ${\rm B}^\nabla_{\rm cris}(R_n)^{\CG_{R_n }}=\BB_{\rm cris}(A_n, A^+_n)^{\CG_{R_n}}=\br\BQ_p$.

\item[ii)] For   $\lambda>0$, we have 
$
{\mathcal BC}(\lambda)(\Spa(R_n[1/p], R_n)^\diam)=(0).
$
\end{itemize}
\end{lemma}

\begin{proof}
We remark that \cite[Rem. 4.2.8]{PRCJM} sketches the proof of a similar result following a suggestion of Scholze, see also Remark \ref{deJong} below.

Part i):  Note that \cite[Cor. 4.8]{Sh}, based on \cite[Cor. 5.3.7]{Br}, gives a similar conclusion, but in there it is assumed that the ring satisfies \cite[Set-up 3.1]{Sh}.
We give an argument which works in our situation: Set $R=R_n$, $A=A_n$, for simplicity. 
First notice that, as in \emph{loc. cit.}, it is enough to show ${\rm B}^\nabla_{\rm dR}(R)^{\CG_{R}}=\BB_{\rm dR}(A, A^+)^{\CG_{R}}=F$. 
Let $T$ be the set of height $1$ prime ideals of $\overline R$ that lie over $p$. Let $\widehat{\overline{R}_{\frak p}}$ denote the $p$-adic completion of the localization $\overline R_{\frak p}$ of $\overline R$ at $\frak p$. Set $(A_{\frak p}, A^+_{\frak p})=(\widehat{\overline{R}_{\frak p}}[p^{-1}], \widehat{\overline{R}_{\frak p}})$ which also gives a perfectoid pair, cf. \cite[Lem. 6.5]{Sh}. Now observe that, since $\overline R$
is a union of Noetherian normal domains which are finite over $R$, the natural ring homomorphisms
\[
 {\overline R}\to \prod_{\frak p\in T} {\overline{R}_{\frak p}},  \qquad  \widehat{\overline R}\to \prod_{\frak p\in T}\widehat{\overline{R}_{\frak p}},
\]
are injective with $p$-torsion free cokernel, cf. the proof of \cite[Lem. 6.13]{Sh}. As in \cite[Prop. 2.37]{Sh}, it follows that the natural ring homomorphism
\[
\BB_{\rm dR}(A, A^+)\to \prod_{\frak p\in T}\BB_{\rm dR}(A_{\frak p}, A^+_{\frak p})
\]
is injective. Hence, 
\[
\BB_{\rm dR}(A, A^+)^{\CG_{R}}\subset \prod_{\frak p\in T}\BB_{\rm dR}(A_{\frak p}, A^+_{\frak p})^{\CG_{R}(\frak p)},
\]
where $\CG_{R}(\frak p)=\Gal(\overline{R}_{\frak p}[p^{-1}]/R[p^{-1}])$ is identified with the (decomposition) subgroup of $\CG_R$
preserving the prime $\frak p$. 

For each height $1$ prime $\frak q$ of $R$ above $p$, we let $\CK_{\frak q}=\CO_{\CK_{\frak q}}[p^{-1}]$ be the 
fraction field of the completion $\CO_{\CK_{\frak q}}$ of the localization $R_{\frak q}$. Fix an algebraic closure $\overline\CK_{\frak q}$ of $\CK_{\frak q}$. Let $\CO_{\overline\CK_{\frak q}}$ denote the ring of integers of $\overline \CK_{\frak q}$ and denote the $p$-adic completion of $\CO_{\overline\CK_{\frak q}}$ by  $\widehat{\CO}_{\overline\CK_{\frak q}}$.
By \cite[Cor. 4.3]{Oh},
\begin{equation}\label{Ohkuba}
 \BB_{\rm dR}(\widehat{\CO}_{\overline \CK_{\frak q}}[p^{-1}], \widehat{\CO}_{\overline\CK_{\frak q}})^{\Gal(\overline\CK_{\frak q}/\CK_{\frak q})}=F.
\end{equation}
(cf. \cite[Lem. 6.9]{Sh}).

Suppose now $\frak p\in T$ lies above $\frak q$ and corresponds to an $R$-embedding $\overline R\hookrightarrow \CO_{\overline \CK_{\frak q}}$.
The resulting inclusion  $\overline R_{\frak p}\hookrightarrow \CO_{\overline \CK_{\frak q}}$ is $\Gal(\overline\CK_{\frak q}/\CK_{\frak q})$-equivariant,
where $\Gal(\overline\CK_{\frak q}/\CK_{\frak q})$ acts on $\overline R_{\frak p}$ via a corresponding surjection $\Gal(\overline\CK_{\frak q}/\CK_{\frak q})\to \CG_{R}(\frak p)$. By \cite[Lem. 6.6]{Sh}, there is a corresponding $\Gal(\overline\CK_{\frak q}/\CK_{\frak q})$-equivariant inclusion
\[
\BB_{\rm dR}(A_{\frak p}, A^+_{\frak p})\hookrightarrow \BB_{\rm dR}(\widehat{\CO}_{\overline \CK_{\frak q}}[p^{-1}], \widehat{\CO}_{\overline\CK_{\frak q}}).
\] 
Hence, (\ref{Ohkuba}) above gives
\[
\BB_{\rm dR}(A_{\frak p}, A^+_{\frak p})^{\CG_R(\frak p)}=F.
\]
(cf.  \cite[Lem. 6.10]{Sh}). The result now follows.

Part ii): If $\lambda=r/s>0$, with $(r,s)=1$, write $E_{r/s}$ for the $s$-dimensional simple isocrystal of slope $r/s$ over $\br\BQ_p$. Then
\[
{\mathcal BC}(\lambda)(\Spd(A_n, A^+_n))=(E_{r/s}\otimes_{\br\BQ_p}\BB_{\crys}(A_n, A^+_n))^{\phi=1}=(\BB_{\crys}(A_n, A^+_n))^{\phi^s=p^r}
\]
(cf. \cite[Thm. 15.2.5]{SW}). 
Hence, since by (\ref{descent1}),
\[
{\mathcal BC}(\lambda)(\Spa(R_n[1/p], R_n)^\diam)\subset ({\mathcal BC}(\lambda)(\Spd(A_n, A^{+}_n)))^{\CG_{R_{n}}},
\]
we have
\[
{\mathcal BC}(\lambda)(\Spa(R_n[1/p], R_n)^\diam)\subset(\BB_{\crys}(A_n, A^+_n)^{\CG_{R_{n} }})^{\phi^s=p^r}.
\]
By part (i), $\BB_{\crys}(A_n, A^+_n)^{\CG_{R_{n} }}=\br\BQ_p$. Since $ \br\BQ_p^{\phi^s=p^r}=(0)$ for $r\neq 0$, the result follows. 
\end{proof}
 
 \begin{remark}\label{deJong}
 In the above argument, we  reduce the proof  
  to showing a corresponding result for the complete dvrs $\CO_{\CK_{\frak q}}$ which have non-perfect residue field. 
 An alternative proof which reduces to the dvrs with perfect residue field $\CO_K$, where $K$ ranges over all finite extensions of $F$, can be given
by considering the collection of all homomorphisms $R_n\to \CO_K$ and using \cite[Prop. 7.3.6]{deJong} to recover the analogue of \cite[Lem. 6.13]{Sh}.  
 \end{remark}

 By covering $V$ by  $\Spa(R_n[1/p], R_n)$, we see that Lemma \ref{Ginv} (ii) also implies that
${\mathcal BC}(\lambda)(V^\diam)=(0)$, for $\lambda>0$. Since $\CU$ is given as a successive extension of ${\mathcal BC}(\lambda)$, $\lambda>0$, we also obtain $\CU(V^\diam)=(0)$. Now, by (\ref{semidirect}), we have 
\[
\tilde G_b(V^\diam)=\underline{J_b(\BQ_p)}(V^\diam)=C^0(|V^\diam|,  J_b(\BQ_p))=J_b(\BQ_p),
\]
since $|V^\diam|=|V|$ is connected. This completes the proof of Proposition \ref{BC}.
\end{proof}

We can now complete  the proof of the main result, Theorem \ref{mainThm}:   Proposition \ref{BC} applied to $V=U^\times$  implies  
\[
(\iota_r){|_{U^{\times\diam}}}=j\cdot \iota^\times_r,
 \]
 for some 
$j\in J_b(\BQ_p)$. We can then redefine the framing to be $j^{-1}\cdot \iota_r$
so that $\iota^\times_r$  extends.  This completes the argument for extending the original framing.

We have now extended both the shtuka $\sV^\times=\sV(\BL^\times)$ and its framing $\iota_r$, from $U^{\times\diam}$ to $U^\diam$;
this gives an extension of $a^\times: U^{\times\diam}\to {\rm Sht}_K$ to $a: U^{\diam}\to {\rm Sht}_K$ and, by the above, completes
the proof. 
 \end{proof}

\subsection{The proof of Theorem \ref{constantIntro}:} In this section, we  show the result below, which is also Theorem \ref{constantIntro} in the introduction.
It is used in the alternative proof of Corollary \ref{BrodySht}, see Remark \ref{alternativeHyperbolicity}:

\begin{theorem}\label{constant}
Suppose that $S$ is a smooth quasi-projective variety  defined over a finite extension $F/\br E$ such that $S\times_{\Spec(F)}\Spec(\bar F)$ is irreducible and has finite \'etale fundamental group.
Every morphism of diamonds $f: S^\diam\to \Sht_{(G, b, \mu, K)}$ over $\Spd(\br E)$ is constant, i.e. $f$ factors through a point 
\[
S^\diam\to \Spd(F)\to \Sht_{(G, b, \mu, K)}
\]
over $\Spd(\br E)$.
\end{theorem}
\begin{proof}
By an argument as in \S \ref{mainProof} (A), we can quickly reduce to the case $G=\GL_n$, with $K=\GL_n(\BZ_p)$.
By the discussion in the beginning of \S \ref{mainProof} (B1), rigid analytic GAGA, and the rigid analytic Riemann existence theorem of L\"{u}tkebohmert,
we see that 
the pro-\'etale $\BZ_p$-local system $\BL=f^*\BL^{\rm univ}$   obtained by pulling back the universal pro-\'etale $\BZ_p$-local system
over $\Sht_{(\GL_n, b, \mu, K)}$ is indeed defined in the algebraic category, as a pro-\'etale $\BZ_p$-local system over the algebraic variety $S$.

Now fix an algebraic closure $\bar F$ of $F$ and a point $\bar s\in S(\bar F)$. By our assumption, there is a finite field extension $F\subset F'\subset \bar F$ and a finite \'etale cover $\pi: T\to S\times_FF'$ such that $T\times_{F'}\bar F$ is irreducible, $\bar s$ lifts to   an $F'$-rational point $t'$ of $T$, and $\pi^*\BL$ is  isomorphic to the pull-back of the  local system $\BL_0=(t')^*\BL$ over $\Spec(F')$ via the structure morphism $T\to \Spec(F')$.  
Suppose we can show that the composition 
\[
T^\diam\to (S\times_FF')^\diam\to \Sht_{(\GL_n, b, \mu, K)}
\]
factors through a point $\Spd(F')\to \Sht_{(\GL_n, b, \mu, K)}$ over $\Spd(\br E)$. Then, by applying Lemma \ref{con1} (a), we see that $(S\times_FF')^\diam\to \Sht_{(\GL_n, b, \mu, K)}$ factors through $\Spd(F')\to \Sht_{(\GL_n, b, \mu, K)}$. Next, by Lemma \ref{con1} (b), $f: S^\diam\to \Sht_{(\GL_n, b, \mu, K)}$ also factors through a point $\Spd(F)\to \Sht_{(\GL_n, b, \mu, K)}$ and the proof is complete. Hence, by this argument we have reduced the proof to the special case in which $S$ has an $F$-rational point $s\in S(F)$, and the local system $\BL=f^*\BL^{\rm unv}$ is   isomorphic to $g^*\BL_0$, with $\BL_0=s^*\BL$, where $g:S\to \Spec(F)$ is the structure morphism. We will assume we are in this situation from now on.

By Proposition \ref{propdeRham}, $\BL_0$ and $\BL$ are horizontal de Rham local systems. In fact, the proof of Proposition \ref{propdeRham} (ii) shows that the analytic bundle $(D_{\rm dR}(\BL), \nabla)$ with connection over $S^{\rm an}$ associated to $\BL$ by the $p$-adic Riemann-Hilbert functor admits a trivialization, i.e. an isomorphism   
\[
\alpha: (V\otimes_{\BQ_p}\CO_{S^{\rm an}}, 1\otimes d)\xrightarrow{\sim}  (D_{\rm dR}(\BL), \nabla)
\]
which is constructed by using the framing of the corresponding shtuka. There is a similar isomorphism  $\alpha_0: V\otimes_{\BQ_p}F \xrightarrow{\sim} D_{\rm dR}(\BL_0)$ over the point $\Sp(F)$. We also have an isomorphism $c: D_{\rm dR}(\BL)\xrightarrow{\sim}  D_{\rm dR}(g^*\BL_0) \xrightarrow{\sim}  g^*D_{\rm dR}(\BL_0) $ respecting the connections. The composition $g^*(\alpha_0)^{-1}\circ c \circ \alpha$ is given by a global section $A$ of ${\rm Aut}(V\otimes_{\BQ_p}\CO_{S^{\rm an}})\simeq \GL_n(\CO_{S^{\rm an}})$ whose matrix entries are functions with zero differential, hence $A$ belongs to ${\rm Aut}(V\otimes_{\BQ_p}F)=\GL_n(F)$. 

Consider now the composition
\[
 S^\diam\xrightarrow{ \ f \ }  \Sht_{(\GL_n, b, \mu, K)}\xrightarrow{\pi_{\rm GM}} {\rm Gr}_{\GL_n, \leq \mu}.
\]
By the definition of the period morphism $\pi_{\rm GM}$, the value of $\pi_{\rm GM}\circ f$ at a $\Spa(R, R^+)$-point of $S^\diam$,  is the point of ${\rm Gr}_{\GL_n,\leq \mu}$ obtained as follows: Recall the $\BB_{\rm dR}^+(R^\sharp, R^{\sharp+})$-lattices $\BM$ and $\BM_0$ associated to the shtuka, as in the proof of Proposition \ref{propdeRham}. We use the framing of the  shtuka to obtain a trivialization of  $\BM_0$ 
\[
\beta: V\otimes_{\BQ_p} \BB_{\rm dR}^+(R^\sharp, R^{\sharp+})=\BB_{\rm dR}^+(R^\sharp, R^{\sharp+})^n\xrightarrow{\sim} \BM_0
\]
and then consider the $\BB_{\rm dR}^+(R^\sharp, R^{\sharp+})$-lattice 
\[
\beta^{-1}(\BM)\subset  \BB_{\rm dR}(R^\sharp, R^{\sharp+})^n;
\]  
this defines a point of ${\rm Gr}_{\GL_n, \leq \mu}$.

Recall that, by \cite[\S 7]{ScholzeHodge}, we have $\BM_0=(\CE\otimes_{\CO_{S^{\rm an}}}\CO\BB_{\rm dR}^+)^{\nabla=0}$, with $\CE={\rm D}_{\rm dR}(\BL)$. We now observe that the   trivialization $\beta$ of  $\BM_0$ is induced by the (horizontal) trivialization $\alpha$ of $\CE={\rm D}_{\rm dR}(\BL)$ above; similarly for the local system $\BL_0$ over $\Spd(F)$. Hence, it follows that the composition
\[
S^\diam\xrightarrow{ \ f \ }  \Sht_{(\GL_n, b, \mu, K)}\xrightarrow{\pi_{\rm GM}} {\rm Gr}_{\GL_n, \leq \mu}\times_{\Spd(\BQ_p)}\Spd(F)
\]
is equal to $\Spd(F)^\diam \xrightarrow{ \ f\circ s\ }  \Sht_{(G, b, \mu, K)}\xrightarrow{\pi_{\rm GM}} {\rm Gr}_{\GL_n, \leq \mu}\times_{\Spd(\BQ_p)}\Spd(F)$ 
composed with the action of $A\in \GL_n(F)$. Hence, $\pi_{\rm GM}\circ f$ factors through a point $S^\diam\to \Spd(F)\to {\rm Gr}_{\GL_n, \leq \mu}$.
Since  $\pi_{\rm GM}: \Sht_{(\GL_n, b, \mu, K)}\to {\rm Gr}_{\GL_n, \leq \mu}$ is \'etale and $S$ is geometrically connected, $|f|: |S|\to |\Sht_{(\GL_n, b, \mu, K)}|$ also factors through a closed point
and we can now conclude the proof as above.
\end{proof}

\begin{remark} a) Theorem \ref{constant} relates to the theme of $p$-adic analytic versions of the theorem of the fixed part, see \cite[\S 5]{JV} for a discussion and other results in this direction.

\sk

b) Suppose the pro-\'etale $\Sht_\infty=\varprojlim_{K'} \Sht_{K'}\to \Sht_K$ is   perfectoid in the sense of \cite[Def. 4.3]{ScholzeHodge}. Then, the corresponding diamond is represented by a perfectoid space $P$ over $\br E$, and Theorem \ref{constant} follows quickly from the fact that all morphisms $S^{\rm an}\to P$ are constant, for example, see \cite[Prop. 2.10, Prop. 5.3]{JV}. This perfectoid property holds for the limit of Rapoport-Zink spaces by work of Scholze-Weinstein \cite{SWpDiv}.  It seems reasonable to expect that the perfectoid property is also true for all (limits  of) local Shimura varieties, but this does not seem to be known.
\end{remark}

   \section{Further questions}\label{s:questions}
  We conclude with a few questions. 
  
  \begin{altitemize}

\item[ 1)] Does the extension property in Theorem \ref{mainThm} still hold when $F$ is replaced by a general complete non-archimedean extension of $\br E$ such as $\BC_p$? A similar question can be asked about the hyperbolicity property in Theorem \ref{BrodySht}.

\item[ 2)] We can define on the local Shimura varieties $\CM_{G, b,\mu, K}$ a variant of the Kobayashi semi-distance, as extended to the non-archimedean set-up by Cherry, see 
\cite{cherry-kobayashi}, \cite[Def. 3.21]{JV}.  In this variant, we insist that the Kobayashi chains are given by morphisms  $f: \BB_{F}\to \CM_{G, b,\mu, K}$ defined over   (varied)  finite field extensions $F$ of $\br E$ and obtain a semi-distance
\[
d: \CM_{G, b,\mu, K}^{\rm class}\times \CM_{G, b,\mu, K}^{\rm class}\to \BR_{\geq 0}\cup\{+\infty\}
\]
on the set of classical points $\CM_{G, b,\mu, K}^{\rm class}$. One can then ask if $d$ is a distance, i.e. if $d(x,y)=0$ only when $x=y$. If this holds, it would be reasonable to say that $\CM_{G, b,\mu, K}$ is ``$p$-adically Kobayashi hyperbolic".

\item[ 3)] The extension theorem and the hyperbolicity properties shown in \cite{Borel}, \cite{OSZP} and in the current paper, are a common feature of global and local Shimura varieties, and also hold for the moduli spaces of $p$-adic shtukas. We are wondering if corresponding properties also hold for the function field analogues of these spaces: In the global case, these analogues are the moduli spaces of global shtukas over function fields introduced by Drinfeld, and studied by L. and V. Lafforgue, Genestier, Varshavsky,  and many others. In the local case, we have the moduli spaces of local shtukas considered by Hartl and Hartl-Viehmann. For example, one can ask if every rigid analytic morphism from the punctured disk over a completion of the function field to one of these spaces extends to  a morphism of the disk to the space (in the local case) or to a natural compactification (in the global case). (The question in the global case is, of course, somewhat imprecise.)
 \end{altitemize}

  \renewcommand{\bibliofont}{\small}

\end{document}